\documentclass[a4paper,12pt]{article}

\usepackage{a4}
\usepackage{amsmath,amssymb,amsthm,accents}
\usepackage{enumitem,ifthen,xcolor}

\newtheorem{theorem}{Theorem}[section]

\newtheorem{lemma}[theorem]{Lemma}
\newtheorem{proposition}[theorem]{Proposition}

\theoremstyle{definition}
\newtheorem{definition}[theorem]{Definition}
\newtheorem{remark}[theorem]{Remark}
\newtheorem{observation}[theorem]{Observation}
\newtheorem{example}[theorem]{Example}

\DeclareSymbolFont{bbold}{U}{bbold}{m}{n}
\DeclareSymbolFontAlphabet{\mathbbold}{bbold}

\newcommand{\C}{\mathbb{C}}

\newcommand{\N}{\mathbb{N}}
\newcommand{\R}{\mathbb{R}}

\newcommand{\cA}{\mathcal{A}}

\newcommand{\cL}{\mathcal{L}}

\newcommand{\Q}{\mathcal{Q}}

\newcommand{\h}{\mathfrak{h}}
\newcommand{\U}{\mathfrak{U}}

\newcommand{\mkU}{U\mkern-1mu}
\newcommand{\mkUhat}{\Uhat\mkern-1mu}
\newcommand{\shift}[1]{\mkern-0.5mu#1\mkern1mu}
\newcommand{\Up}{{U\mkern-3mu,\mkern1mup}}
\newcommand{\lrho}{{\!\rho}}
\newcommand{\chic}{\chi_{\rm c}}
\newcommand{\smedmuskip}{\mskip 2.0mu plus 1.0mu minus 2.0mu}
\newcommand{\scirc}{\smedmuskip{\circ}\smedmuskip}
\newcommand{\bplusb}{b_1\mkern-1mu\smedmuskip{+}\smedmuskip b_2}
\newcommand{\tomega}{\widetilde\omega}
\newcommand{\homega}{\widehat\omega}
\newcommand{\Uhat}{\widehat U}
\newcommand{\Utilde}{\widetilde U}

\newcommand{\sUtilde}{\smash{\Utilde}}
\newcommand{\Bhat}{\widehat B}
\newcommand{\s}{{\rm s}}
\renewcommand{\a}{{\rm a}}

\numberwithin{equation}{section}

\newcommand{\from}{\colon}

\newcommand{\ind}{\mathbbold{1}}
\newcommand{\snabla}{\nabla\mkern-1mu}

\renewcommand{\div}{\operatorname{div}}
\newcommand{\sgn}{\operatorname{sgn}}

\renewcommand\le{\leqslant}
\renewcommand\ge{\geqslant}
\renewcommand\subset{\subseteq}
\renewcommand\supset{\supseteq}

\newcommand{\set}[2]{\bigl\{#1{:}\;#2\bigr\}}
\newcommand{\sset}[2]{\{#1{:}\;#2\}}
\newcommand{\restrict}[1]{{\restriction}_{#1}}

\newcommand{\inter}[1]{\accentset{\smash{\raisebox{-0.15ex}{$\scriptstyle\circ$}}}{#1}\rule{0pt}{2.3ex}}
\newcommand{\sinter}[1]{\smash{\inter{#1}}}

\newcommand{\gen}{\mathcal{L}}

\newcommand{\loc}{{\rm loc}}
\renewcommand{\c}{{\rm c}}
\newcommand{\rfrac}[2]{\tfrac{#1}{\raisebox{0.1em}{$\scriptstyle #2$}}}
\newcommand{\lfrac}[2]{\frac{\raisebox{-0.15em}{$#1$}}{#2}}
\newcommand{\twop}{\frac{2}{\raisebox{0.07em}{$\scriptscriptstyle p$}}}
\newcommand\pstrich{{\raisebox{0.15ex}{$\scriptstyle p\mkern0.2mu$}'}}

\newcommand{\RRe}{\operatorname{Re}}
\newcommand{\IIm}{\operatorname{Im}}

\newcommand{\<}{\langle}
\renewcommand{\>}{\rangle}
\newcommand{\eps}{\varepsilon}

\newcommand{\form}{\mathfrak{t}}
\newcommand{\amax}{\mathfrak{a}}
\newcommand{\dual}[2]{\bigl(#1,#2\bigr)}

\newcommand{\inner}[2]{\bigl\<#1,#2\bigr\>}
\newcommand{\+}{\hspace{0.08em}}
\newcommand{\olwn}{\mkern1mu\overline{\mkern-1mu w_n \mkern-1.5mu}\mkern1.5mu}
\newcommand{\chiind}[1]{\ind_{ \left[|u|^{p-2} #1^2\rule{0ex}{1.3ex}\right] }}

\newcommand{\avoidbreak}{\postdisplaypenalty50}

\newcounter{teller}
\newcounter{tellerint}
\newenvironment{Rmlist}{\begin{list}
{\rm (\Roman{tellerint})\hfill}
{\usecounter{tellerint}\setcounter{tellerint}{\value{teller}}\addtocounter{teller}{1}%
\leftmargin=12mm \labelwidth=12mm%
\labelsep=0cm \parsep=0cm}}{\setcounter{teller}{\value{tellerint}}\end{list}}

\newlength{\widthoftheindent}
\newcounter{alpteller}

\makeatletter
\newcommand{\dnepbbl}{\@beginparpenalty=0}  
\makeatother
\newenvironment{alplist}[1][]{%
\dnepbbl
\begin{list}%
{\ifthenelse{\equal{#1\arabic{alpteller}}{1}}
            {\rm (\alph{alpteller}) }
            {\hskip\widthoftheindent\llap{\rm(\alph{alpteller})} }}
{\usecounter{alpteller} \leftmargin=0cm
\labelwidth=0pt \labelsep=0pt \itemindent=0pt
\parsep=0cm \topsep=0cm \itemsep=0cm \listparindent=\parindent}}
{\end{list}}

\makeatletter
\@ifundefined{MoveEqLeft}{%
\newcommand\MoveEqLeft[1][2]{%
  \global\@tempdima=#1em%
  \kern\@tempdima%
  &
  \kern-\@tempdima}
}{}
\makeatother
\newcommand\alignstrut{\rule{0pt}{2.7ex}}

\begin{document}

\belowdisplayshortskip=\belowdisplayskip
\newcommand\smallbds{\belowdisplayskip=6.5pt plus3.5pt minus3pt
                     \belowdisplayshortskip=\belowdisplayskip}

\title{On the $L^p$-theory of $C_0$-semigroups \\ 
associated with second-order elliptic operators \\
with complex singular coefficients}
\author{A.F.M. ter Elst, Vitali Liskevich\footnote{%
After the main results of this paper were obtained, Vitali Liskevich passed away.
Then the remaining authors rearranged and polished the manuscript.}\:,
Zeev Sobol, Hendrik Vogt}

\date{}
\maketitle

\begin{abstract}
We study $L^p$-theory of second-order elliptic divergence type
operators with complex measurable coefficients. The major aspect
is that we allow complex coefficients in the main part of the operator, too.
We investigate generation of analytic
$C_0$-semigroups under very general conditions on the
coefficients, related to the notion of form-boundedness.

We determine an interval $J$ in the $L^p$-scale, not necessarily containing $p=2$,
in which one obtains a consistent family of quasi-contractive semigroups.
This interval is close to optimal, as shown by several examples.
In the case of uniform ellipticity we construct a family of
semigroups in an extended range of $L^p$-spaces, and we
prove $p$-independence of the analyticity sector and of the spectrum
of the generators.

\vspace{8pt}

\vspace{2pt}

\end{abstract}

\section{Introduction and main results}\label{intro}

The aim of this paper is to develop the $L^p$-theory of second-order 
elliptic differential operators on an open set
$\Omega\subset\R^N$, where $N\in\N$, corresponding to the
formal differential expression
\[
  \cL = -\snabla \cdot(\cA\nabla ) + b_1\cdot\nabla + \snabla \cdot b_2 + Q,
\]
with \emph{complex} measurable coefficients $\cA\colon
\Omega\to \C^{N\times N}$,\, $b_1,b_2 \colon \Omega\to\C^N$ and
$Q \colon \Omega\to\nobreak\C$.
We present rather general sufficient conditions on the operator~$\cL$ under which,
for $p$ in some subinterval of $[1,\infty)$,
one can construct a $C_0$-semi\-group~$S_p$
on $L^p:=L^p(\Omega)$ whose generator is associated with $\cL$ in a natural way
(which will be made precise below).
Moreover, we give sufficient conditions for quasi-contractivity and for analyticity of this semigroup.

The approach via $C_0$-semigroups is one of the
traditional approaches to solving the initial value problem for
parabolic partial differential equations in different Banach function spaces.
Various examples are well documented in, e.g., \cite{D,EN,P},
and in some monographs \cite{Amann,lunardi} semigroups are the main tool for the study of parabolic equations.

A classical approach to the construction of $C_0$-semigroups
on $L^p$ generated by divergence type second-order elliptic operators
with measurable coefficients is the form method in $L^2$.
The sesquilinear form corresponding to $\cL$ is given by
\[
  \form (u,v)
  =   \int_\Omega \inner{\cA \nabla u}{\nabla v}
    + \int_\Omega\inner{b_1\cdot \nabla u}{v}
    - \int_\Omega\inner{b_2\+u}{\nabla v}
    + \int_\Omega\inner{Q\+u}{v},
\]
where the domain $D(\form)$ will be specified below; cf.~\eqref{form-def}.
Here, $\<\cdot,\cdot\>$ denotes the scalar product in $\C^N$ resp.~$\C$ as appropriate.

If $\form$ is a densely defined closed sectorial form, then by the Kato representation theorem \cite[Theorem~VI.2.1]{kato80}
there is a unique $m$-sectorial operator $\gen$ associated with the form
$\form$, which is minus the generator of an analytic quasi-contractive
$C_0$-semigroup $S = (e^{-t\gen})_{t\ge0}$ on $L^2$ (see also the recent paper
\cite{AE1} for a modification of the approach that applies to \emph{non-closable} sectorial forms).
In order to construct the $C_0$-semigroup on $L^p$, one looks at the restrictions
$S(t)\restriction{L^2\cap L^p}$ and studies
whether they can be extended to bounded operators on $L^p$ forming a $C_0$-semigroup $S_p$.
Then, by construction, $S_p$ is consistent with the semigroup~$S$.

A relatively simple case is when the semigroup $S$ leaves invariant the unit balls
of both $L^1$ and $L^\infty$. This property can be verified using form techniques
(see, e.g., \cite[Theorems~2.13, 2.14 and Corollary~2.16]{Ou} for the criteria).
By interpolation, $S$ can then be extended to a consistent family
of $C_0$-semigroups $S_p$ of contractions on $L^p$, $1\le p\le\infty$
(with weak$^*$-continuity on $L^\infty$).

Another possible approach to constructing a consistent family of $C_0$-semi\-groups in the $L^p$-scale
is via integral kernel estimates.
If $S$ has a kernel satisfying Gaussian upper bounds (or Poisson upper bounds),
then one easily sees that $S$ extends to a $C_0$-semigroup on $L^p$ for all $p \in [1,\infty)$.
Suppose, e.g., that $\Omega=\R^N$, that all the coefficients of the operator are bounded
and that the matrix $\cA$ is uniformly elliptic.
If the coefficients are real-valued, then Gaussian bounds are valid by
the classical result of Aronson \cite{aro68}.
More recently it was shown that Gaussian bounds hold if merely
the coefficients of the principle part of~$\gen$ are real-valued,
or if they are uniformly continuous; see \cite{Aus1},~\cite[Theorems~6.10 and~6.11]{Ou}.
In dimensions $N = 1$ and $N = 2$ these additional assumptions are not
needed, by \cite{AMT1}.

If the coefficients of the principle part are complex-valued, then in general
Gaussian kernel bounds are no longer valid:
in \cite[Proposition~3]{ACT} it is shown that for $N\ge5$ there exists a complex
uniformly elliptic Gilbarg--Serrin matrix $\mathcal{A}$ such that the
$C_0$-semigroup $S$ on $L^2(\R^N)$ generated by $\snabla\cdot\cA\nabla$
does not extrapolate to a semigroup of bounded operators on $L^\infty$.
Using \cite[Proposition~2.10]{HMM} and a scaling argument one obtains a
stronger result: given $N\ge3$ and $p\notin[\frac{2N}{N+2},\frac{2N}{N-2}]$,
one can choose $\cA$ such that $S$ does not extrapolate to a semigroup
of bounded operators on $L^p$.

Even if the coefficients of $\cL$ are real-valued,
the presence of singular lower-order terms can
result in a restricted range of~$p$ for which one obtains
a semigroup on $L^p$ corresponding to $\cL$.
A typical example of such a situation is the
Schr\"odinger semigroup with negative Hardy potential, which corresponds to the Cauchy problem
\begin{equation}\label{S-H}
u_t = \Delta u + \beta\tfrac{(N-2)^2}{4} |x|^{-2}u,
\quad u(0) = f \in L^1(\R^N) \cap L^\infty(\R^N),
\end{equation}
where $\beta\in (0,1)$.
To describe the existence of the $C_0$-semigroup on $L^p(\R^N)$,
we assume that $N\ge 3$.
The $L^2$-semigroup corresponding to~\eqref{S-H}
can be extrapolated to a contraction semigroup on $L^p$ for all
$p\in [p_-,p_+]$, where $p_\mp = \frac2{1\pm\sqrt{1-\beta}}$
(see \cite[Theorem~1]{BS}),
to a non-quasi-contractive $C_0$-semigroup for all
$p\in (p_{\min},p_-) \cup (p_+,p_{\max})$,
with $p_{\max} = \frac{N}{N-2}p_+$ and $p_{\min}=\frac{p_{\max}}{p_{\max}-1}$
(cf.\ \cite[Theorem~6.1(b)]{kps81}),
whereas for all $p\in[1,\infty)$ outside of the interval $(p_{\min},p_{\max})$ there is
no $C_0$-semigroup on $L^p$ corresponding to~\eqref{S-H},
and the Cauchy problem is not well-posed (see \cite[Section~4]{lsv02}).

Another example comes from the non-symmetric situation of divergence type operators
with singular drift terms (see \cite{KS,L}):
again one obtains a quasi-contractive semigroup $S_p$ only for $p$
in a certain subinterval of $[1,\infty)$.

Quasi-contractivity of $S_p$ is of course intimately related to (formal) accretivity
of~$\cL$ in~$L^p$ (or formal dissipativity of $-\cL$).
For the general differential expression~$-\cL$, the question of
formal dissipativity in $L^p$ was studied in \cite{CM2006},
and a necessary and sufficient condition was obtained. One of the
main results of the present paper (Theorem~\ref{mainresult2}) shows that a very close
condition guarantees existence of the corresponding quasi-contractive
$C_0$-semigroup on $L^p$. In particular cases our condition coincides with the
condition of \cite{CM2006}, namely when the imaginary part $\cA_1$ of $\cA$ is
anti-symmetric, or when the lower-order terms are absent and $\cA_1$ is symmetric.
See also Remark~\ref{mainresult-rem}\ref{mainresult-rem(e)}.

In previous work, the question of generation of $C_0$-semi\-groups on~$L^p$
by elliptic operators with complex coefficients has only been studied
under restricted assumptions.  In \cite[Section~5]{CM2006}, e.g.,
some smoothness of the coefficients is assumed, and in \cite[Chapter~4]{Ou},
only the case of anti-symmetric $\cA_1$ is considered.

\bigskip

In general the form $\form$ need not be sectorial, so that
Kato's representation theorem cannot be used directly.
The initial approach to studying this case
was by approximating the singular coefficients by bounded ones
(cf.\ \cite[Theorem~6]{L}).
A more powerful approach was developed in \cite{sv02},
where general second-order formal differential expressions with \emph{real-valued} coefficients were studied,
and the natural functional responsible for the accretivity in $L^p$ was identified.
A positive potential $U$ was introduced that `absorbs' all the
singularities of the lower-order terms of the differential expression,
in the sense that, being added to the corresponding sesquilinear form,
it makes the resulting form sectorial in~$L^2$.
Then a quasi-contractive $C_0$-semi\-group on $L^p$ was constructed by an approximation
procedure that removes the added potential.
A crucial ingredient in the realization of this approximation idea was
the perturbation theory for positive semigroups developed in \cite{voi86, voi88}.

The main result of \cite{sv02} establishes the precise interval in the $L^p$-scale
where there exists a \emph{quasi-contractive} $C_0$-semi\-group
corresponding to the formal differential expression~$\cL$.
In \cite{lsv02} it was shown that in the case of a uniformly elliptic principal part of~$\gen$,
the interval in the $L^p$-scale can be extended if one allows
for non-quasi-contractive semigroups,
and an example was given where this extended interval is
the maximal interval of existence of the semigroup.
The main technique of \cite{lsv02} is the technique of weighted estimates,
which provides additional information such as $p$-independence
of the sector of analyticity of the semigroups
and $p$-independence of the spectrum of the generators.

In this paper we pursue the same goals as in the previous papers \cite{sv02,lsv02},
with the significant difference that the coefficients are allowed to be complex-valued.
As in \cite{sv02} we do not assume sectoriality of the form associated with $\cL$,
so we follow the idea described above: add an auxiliary positive potential $U$ to the form
to make it sectorial, and remove $U$ in $L^p$ by approximation
to obtain a quasi-contractive $C_0$-semigroup on $L^p$ corresponding to $\cL$.
The range of $p$ for which this is possible is determined by a
family of functionals $\tau_p$
(computed from the coefficients of $\cL$) that we specify in~\eqref{tau_p}.

We point out that the main tool of \cite{sv02}
-- namely, the perturbation theory for positive semigroups --
is not applicable in the presence of non-real coefficients and other tools must be used.
As a result, we need stronger assumptions on the auxiliary potential $U$;
cf.\ Remark~\ref{rem1}\ref{rem1(d)}.

In the case of a uniformly elliptic principal part of $\cL$,
we use the technique of weighted estimates as in \cite{lsv02}
to extend the interval in the $L^p$-scale of existence of the semigroup.
The $C_0$-semi\-groups thus obtained are analytic with
$p$-independent sector of analyticity and spectrum of the generators.

\bigskip

Before formulating and discussing the assumptions on the coefficients of $\cL$,
we introduce some notation.
Throughout, all the function spaces will consist of complex-valued functions.
By $\<\xi,\eta\> = \xi \cdot \overline\eta$ we denote the inner product of $\xi,\eta \in \C^N$.
If $f,g \colon \Omega \to \C$ or $f,g \colon \Omega \to \C^N$ are measurable functions,
then we define
$(f,g) := \int_\Omega f \cdot \overline g$
(integration with respect to Lebesgue measure)
whenever $f \cdot \overline g \in L^1$.
Given $p \in [1,\infty]$, we let $p'$ denote the dual exponent,
$\frac{1}{p'}+\frac 1p = 1\rule[-1.3ex]{0pt}{0pt}$.

Let $U \colon \Omega \to \C$ be a measurable function.
Then we also consider $U$ as a sesquilinear form in $L^2$ with domain
$\Q(U) = \sset{ u \in L^2 }{ U |u|^2 \in L^1 }$ by setting $U(u,v) =
\int U u \+ \overline v$ for all $u,v \in \Q(U)$.
If $\h$ is a sesquilinear form in $L^2$, then the sum $\h+U$
is defined by $(\h+U)(u,v) := \h(u,v) + U(u,v)$ on $D(\h+U) := D(\h)\cap \Q(U)$.
In the case of symmetric $\h$ and real-valued $U$
we say that $\mkU \le \h$ if $D(\h) \subset \Q(U)$ and $U(u) \le \h(u)$ for
all $u\in D(\h)$.

For given $p,q\in[1,\infty]$ and a linear operator $B\from L^p\to L^q$ we denote its norm by
$\|B\|_{p\to q}$.

\medskip

In the following we formulate the qualitative assumptions on the operator $\cL$.
For this,
we decompose $\cA=\cA_0+i\cA_1$, where $\cA_0,\cA_1 \colon \Omega\to\R^{N\times N}$.
For $j=0,1$ we denote by $\cA_j^\s$ and $\cA_j^\a$ the symmetric and
anti-symmetric parts of $\cA_j$, respectively, i.e.,
\[
  \cA_j^\s=\tfrac12(\cA_j^{}+\cA_j^\top) \quad \textrm{and} \quad
  \cA_j^\a=\tfrac12(\cA_j^{}-\cA_j^\top).
\]
For the matrix $\cA_0^\s$ we assume that

\begin{Rmlist}
\item \label{assump1} $\cA_0^\s \in L^1_\loc$,
    $\cA_0^\s(x)$ is positive definite for a.e.\ $x\in\Omega$,
    and $(\cA_0^\s)^{-1}\in L^1_\loc$\,.
\end{Rmlist}

We define the symmetric form $\amax$ in $L^2$ by
\begin{equation}\label{D-amax-def}
  D(\amax) := \set{ u\in W^{1,1}_\loc \cap L^2 }
                  { \< \cA_0^\s\nabla u, \nabla u\> \in L^1 }
  \smallbds
\end{equation}
and
\begin{equation}\label{amax-def}
  \amax(u,v) := \dual{\cA_0^\s\nabla u}{\nabla v}
              = \int_\Omega \<\cA_0^\s \nabla u, \nabla v\>.
  \smallbds
\end{equation}
It follows from Assumption~\ref{assump1} that the form
$\amax$ is a Dirichlet form in $L^2$. (For the closedness
cf.\ \cite[Theorem~3.2]{RW}.
Moreover, $\amax$ is densely defined since $\cA_0^\s \in L^1_\loc$ implies that $C_\c^1 \subset D(\amax)$.)

We decompose the potential $Q$ as $Q = V\mkern-2mu + i W$, where $V$ and
$W$ are real-valued, and we define $h_{\max} := \amax + V^+$, i.e.,
\[
  h_{\max}(u,v)
  = \dual{\cA_0^\s \nabla u}{ \nabla v} + V^+(u,v)
\]
for all $u,v \in D(h_{\max}) = D(\amax) \cap \Q(V^+)$.
Then $h_{\max}$ is a closed positive symmetric form in $L^2$
as a sum of two such forms,
and $h_{\max}$ satisfies the Beurling--Deny criteria.
Note that $D(h_{\max})$ is dense in $L^2$ if $Q \in L^1_\loc$.
We fix a restriction $h_0$ of $h_{\max}$ and assume that
\begin{Rmlist}
\item \label{assump2} 
$h_0$ is a Dirichlet form in $L^2$.
\end{Rmlist}
This means in particular that $h_0$ is densely defined.
The form $h_0$ will be the main reference object in the perturbation
approach we develop in this paper.
An alternative reference form was used in \cite{sv02}; see Remark~\ref{reference-form}.

The next two assumptions on the coefficients of $\cL$ are needed to define the form $\form$ on $D(h_0)$:
\begin{Rmlist}
\item \label{assump3} $|\< \cA \xi, \eta \>|^2
  \le c^2 \+ \< \cA_0^\s \xi,\xi\> \+ \< \cA_0^\s \eta,\eta\>$
  for all $\xi,\eta\in\C^N$, for some $c>0$,
\item \label{assump4} $ \<(\cA_0^\s)^{-1} b_1, b_1\> +
    \<(\cA_0^\s)^{-1} b_2, b_2\> + |Q| \le c(h_0+1)$ for some $c>0$.
\end{Rmlist}
Throughout this paper we assume that Assumptions \ref{assump1}~--~\ref{assump4}
are satisfied.
\pagebreak[2]

It follows from~\ref{assump1} and~\ref{assump3} that $\cA\in L^1_\loc$.
By~\ref{assump3}, $(u,v)\mapsto \dual{\cA\nabla u}{\nabla v}$
is a bounded sesquilinear form on $D(h_0)$.
By~\ref{assump4} and the Cauchy--Schwarz inequality, the forms
$(u,v)\mapsto \dual{b_1 \cdot \nabla u}{v}$, $(u,v)\mapsto \dual{b_2\+u}{\nabla v}$
and $Q$ are bounded sesquilinear forms on $D(h_0)$.
Thus we can define a bounded sesquilinear form $\form$ on $D(h_0)$ by
\begin{equation}\label{form-def}
\form(u,v)
= \dual {\cA \nabla u}{\nabla v}
   + \dual{b_1 \cdot \nabla u}{v}
   - \dual{b_2\+u}{\nabla v}
   + \dual{Q\+u}{v}.
\end{equation}

The domain $D(h_0)$ of $\form$ will determine a particular realization of $-\gen$
as a generator of a $C_0$-semigroup. This is analogous to the classical theory
of elliptic and parabolic differential equations in divergence form with smooth
coefficients, where (homogeneous) boundary conditions are used.
For example, the form $h_0=\overline{h_{\max}\restrict{C_\c^\infty(\Omega)}}$
defines (homogeneous) Dirichlet boundary conditions on $\partial\Omega$,
and, if $\Omega$ has Lipschitz boundary, then taking
$h_0=h_{\max}$ leads to the generalized Neumann boundary condition
$\inner{\cA \nabla u - b_2u}{n}=0$ on $\partial \Omega$,
where $n$ is the outer normal to~$\partial \Omega$.

Our first main result, Theorem~\ref{mainresult} below,
involves several explicit parameters that are defined in a somewhat technical fashion;
we will need the constants from the next lemma.

\begin{lemma} \label{lcomplex1.05}
Adopt the above assumptions and notation.
\begin{alplist}[i]
\item  \label{lcomplex1.05(a)}
There exists an $\alpha_\s \ge 0$ such that
\begin{equation}\label{alpha-s}
  |\<\cA_1^\s \xi,\eta\>|^2
  \le \alpha_\s^2 \+ \<\cA_0^\s \xi,\xi\> \+ \<\cA_0^\s \eta,\eta\>
\end{equation}
for all $\xi,\eta\in \R^N$.
\item  \label{lcomplex1.05(b)}
There exist $\beta',B' \ge 0$ such that
\begin{equation}\label{beta-'}
\begin{split}
\MoveEqLeft[4]
\IIm \dual{\cA_1^\a \nabla u -  u\IIm(\bplusb)}{\nabla u} \\
 &\le \bigl(\beta'^2h_0(|u|)+B'\|u\|_2^2\bigr)^{\frac12} \dual{\cA_0^\s \eta(u)}{\eta(u)}^{\frac12}
\end{split}
\end{equation}
for all $u \in D(h_0)$, where
$\eta(u) := \IIm(\overline{\sgn u}\,\nabla u)$.
\item  \label{lcomplex1.05(c)}
There exist $\beta_1,\beta_2,B_1,B_2 \ge 0$ such that
\begin{equation}\label{beta-1-2}
 (-1)^j \dual{(\RRe b_j)u}{\nabla u} \le \beta_j h_0(u) + B_j\|u\|_2^2
\end{equation}
for all $0 \le u \in D(h_0)$ and $j=1,2$.
\item  \label{lcomplex1.05(d)}
There exist $\gamma,\Gamma \ge 0$ such that
\begin{equation}\label{gamma}
V^- = (\RRe Q)^- \le \gamma \+ h_0+\Gamma.
\end{equation}
\end{alplist}
\end{lemma}
\begin{proof}
\ref{lcomplex1.05(a)}
By~\ref{assump3} there exists an $\alpha_\s \ge 0$ such that
\[
  |\<\cA_1^\s \xi,\xi\>| = \mathopen|\IIm \<\cA \xi,\xi\>|
  \le \alpha_\s \+ \<\cA_0^\s \xi,\xi\>
  \avoidbreak
\]
for all $\xi\in\R^N$.
This implies~\eqref{alpha-s} since $\cA_1^\s$ is real-valued and symmetric.

\ref{lcomplex1.05(b)}
Using the anti-symmetry of $\cA_1^\a$, one easily computes
\begin{equation}\label{elcomplex1.05(b);1}
  \IIm \dual{\cA_1^\a \nabla u - u\IIm(\bplusb)}{\nabla u}
  = \dual{-2\cA_1^\a \nabla |u| + |u|\IIm(\bplusb)}{\eta(u)}
\end{equation}
for all $u \in D(h_0)$; then~\eqref{beta-'} follows from \ref{assump3} and \ref{assump4}.

\ref{lcomplex1.05(c)} and \ref{lcomplex1.05(d)} are easy consequences of~\ref{assump4}.
\end{proof}

The combination of coefficients in~\eqref{beta-'} is explained by the circumstance
that in essence, $\cA_1^\a$ is the coefficient of an imaginary drift
$i(\div\cA_1^\a)\cdot\nabla$; cf.\ Example~\ref{imag-coeffs}.
For further comments on the inequalities~\eqref{alpha-s}~--~\eqref{beta-1-2}
see Remark~\ref{const-rem2}.

In addition to the above constants we will use the following notation:
\begin{align}
\delta_p &:= \alpha_\s\left|1-\frac2p\right| + \frac{\beta'}{2}, \nonumber \\[0.8ex]
\eps_p &:= \frac{4}{pp'} - \frac2p\beta_1 - \frac2{p'}\beta_2 - \delta_p^2 - \gamma, \label{deo} \\[0.8ex]
\homega_p &:= \frac2pB_1 + \frac2{p'}B_2 + \Gamma + \Bhat_p \nonumber
\end{align}
for all $p \in [1,\infty]$, where
\begin{equation}\label{Bhat}
  \Bhat_p :=
  \begin{cases}
  \frac{B'}{4} + \frac{\alpha_\s B'}{2\beta'} \bigl|1-\rfrac2p\bigr| & \text{if}\ \beta'>0, \\[0.4ex]
  \frac{B'}{4} & \text{if}\ \beta'=0.
  \end{cases}
\end{equation}
(In Theorem~\ref{mainresult} we will require
$\alpha_\s B' = 0$ in the case $\beta'=0$.)
Then the set
\begin{equation}\label{I-def}
  I := \set{p\in[1,\infty)}{\eps_p\ge0}
\end{equation}
is a closed interval since $\rfrac1p \mapsto \eps_p$ is a
concave function as the minimum of two quadratic polynomials with
negative leading coefficients.

In order to state the first main result, we need to 
formulate in which sense a $C_0$-semigroup on $L^p$ is associated with the form $\form$.
We define the set $\U$ of potentials as
the set of all measurable functions $U\from\Omega\to [0,\infty)$
satisfying the following two properties:
\begin{itemize}[itemsep=0ex,topsep=0.5ex plus 0.2ex minus 0.1ex,parsep=0.4ex]
\item
there exists a $C > 0$ such that $\mkU\le C(h_0+1)$,
and,
\item
$\form + U$ is a closable sectorial form.
\end{itemize}
Note that $\U$ depends on both $\form$ and $h_0$.
If $\mkU \in \U$ then we denote by $\gen_U$ the $m$-sectorial operator in $L^2$
associated with the closure of $\form+U$ and by
$S_U$ the $C_0$-semigroup on $L^2$ generated by $-\gen_U$.
We shall show that both the conditions of 
Theorem~\ref{mainresult} and of Theorem~\ref{mainresult2}
imply that $\U \neq \varnothing$.

\begin{definition}\label{asso}
\dnepbbl
Let $p \in [1,\infty)$.
We say that a $C_0$-semigroup $S_p$ on $L^p$ is associated with
the couple $(\form,\U)$ if the following two properties hold:
\begin{itemize}[itemsep=0ex,topsep=0.5ex plus 0.2ex minus 0.1ex,parsep=0.4ex]
\item
$\U\ne\varnothing$, and for each $\mkU\in\U$
the semigroup $S_U$ on $L^2$ extrapolates to a $C_0$-semigroup $S_\Up$ on $L^p$;
\item
if $U,U_1,U_2,\ldots\in\U$ satisfy $U_m\le U$ for all $m \in \N$ and
$U_m\to 0$ a.e., then $S_{U_m\shift,p} \to S_p$ \emph{in the semigroup sense},
i.e., for all $f\in L^p$ and $T>0$ one has
\[
  S_{U_m\shift,p}(t)f \to S_p(t)f
  \mbox{ in }L^p \mbox{ as }m\to\infty, \mbox{ uniformly for }t\in[0,T].
\]
\end{itemize}
\end{definition}

\begin{remark}\label{rem1}
\begin{alplist}
\item \label{rem1(b)} 
If $\mkU\in\U$ then $(\mkU-c)^+\in\U$ for all $c>0$,
and in Definition~\ref{asso} one can choose, e.g.,
$U_m = (\mkU - m)^+$ for all $m \in \N$.

\item \label{rem1(c)} 
Assume that the form $\form$ is sectorial and closable, and let $S$ be
the associated analytic semigroup on $L^2$. 
Let $p \in [1,\infty)$ and assume that there exists a $C_0$-semigroup $S_p$
on~$L^p$ associated with the couple $(\form,\U)$, in the sense 
of Definition~\ref{asso}.
Then $S_p$ is consistent with $S$ (choose $U = U_m = 0$ in the definition).

\item \label{rem1(d)} 
The approximation method of Definition~\ref{asso} has already
been used in \cite{sv02} and \cite{LM} for constructing semigroups associated
with second-order differential operators. In those papers, the authors
did not need to require form-boundedness of the potentials since they
could work with positive and dominated semigroups, respectively.
\end{alplist}
\end{remark}

The next theorem, which is the first main result of this paper,
shows that $\cL$ corresponds to a quasi-contractive semigroup on $L^p$
for all $p \in I$ (see~\eqref{I-def}).

\begin{theorem}\label{mainresult}
Let Assumptions \ref{assump1}~--~\ref{assump4} be satisfied,
and assume that $\inter{I} \ne \varnothing$.
Suppose that $\beta' > 0$ or $\alpha_\s B' = 0$.
Then $\U \neq \varnothing$ and 
there exists a consistent family of $C_0$-semigroups $S_p = (e^{-t\gen_p})_{t\ge0}$
on $L^p$, $p \in I$, with the following properties.

If $p\in \inter{I}$ then $S_p$ is an analytic semigroup associated with
$(\form, \U)$ in the sense of Definition~\ref{asso}.
Moreover,
for all $u\in D(\gen_p)$ one has $v_p := u|u|^{\frac p2-1}\in D(h_0)$ and
\[
  \RRe\dual{(\homega_p + \gen_p)u}{u|u|^{p-2}}
  \ge \eps_p h_0(|v_p|),
\]
and if $\eps \in (0,1)$ satisfies
$\eps + \frac{\eps}{1-\eps}\delta_p^2 \le \eps_p$, then
\[
  \RRe\dual{(\homega_p + \tfrac{\eps}{1-\eps} \Bhat_p + \gen_p)u}{u|u|^{p-2}}
  \ge \eps h_0(v_p).
\]
Finally,
$
  \|S_p(t)\|_{p\to p} \le e^{\homega_p t}
$
for all $p \in I$ and $t \ge 0$.
\end{theorem}

\begin{remark}\label{mainresult-rem}
\begin{alplist}
\item \label{mainresult-rem(c)} 
The case where $\beta' = 0$ and $\alpha_\s B' > 0$ is not covered in the above result.
In this case one does not necessarily obtain a quasi-contractive $C_0$-semigroup on $L^p$
associated with $(\form,\U)$ if $p \in \partial I$;
see Example~\ref{open-ex}.

\item \label{mainresult-rem(e)} 
Assume that $b_1=b_2=0$, $Q=0$ and $\cA_1^\a=0$. Then $\homega_p=\Bhat_p=0$, so that
Theorem~\ref{mainresult} yields a contractive $C_0$-semigroup $S_p$ on $L^p$
if $\eps_p = \frac{4}{pp'} - \alpha_\s^2(1-\rfrac2p)^2 \ge 0$, i.e., if
\[
  \alpha_\s{|p-2|} \le 2\sqrt{p-1}\,.
\]
In \cite[Theorem~1]{CM2006}, this condition is shown to be a necessary condition for the
formal $L^p$-dissipativity of $-\cL$,
i.e., for the property that for all $u \in C_\c^1(\Omega)$ one has
\vspace{-1.5ex}
\begin{align*}
\RRe \dual{\cA \nabla u}{\nabla(u|u|^{p-2})} &\ge 0 \quad \text{if}\ p\ge2, \\[0.5ex]
\RRe \dual{\cA \nabla(u|u|^{p'-2})}{\nabla u} &\ge 0 \quad \text{if}\ p<2.
\end{align*}

\item \label{mainresult-rem(f)} 
Note that $\infty\notin I$ by definition. Nevertheless,
Theorem~\ref{mainresult} can be used to derive $L^\infty$-properties of the semigroup:
suppose that $\alpha_\s = \beta' = \beta_2 = \gamma = 0$ and $\beta_1 < 2$.
Then $\eps_\infty = 0$ and $\inter{I} \ne \varnothing$.
It follows from Theorem~\ref{mainresult} that one has the bound
$\|S_p(t)\|_{p \to p} \le \exp\bigl((\frac{2}{p}B_1+\frac{2}{p'} B_2 + \Gamma +\frac{B'}{4}) t\bigr)$
for all $t \ge 0$ and large $p$.
Hence there exists a quasi-contractive weak$^*$-continuous semigroup $S_\infty$
of weak$^*$-continuous operators on $L^\infty$ that
is consistent with $S_p$ for all $p \in I$, and
$\|S_\infty(t)\|_{\infty \to \infty} \le \exp\bigl((2 B_2 + \Gamma + \frac{B'}{4}) t\bigr)$
for all $t \ge 0$.
\end{alplist}
\end{remark}

Theorem~\ref{mainresult} will be derived from Theorem~\ref{mainresult2} below.
There we replace the assumption $\inter{I} \ne \varnothing$
with a more general assumption which involves the following notion.
Let $p\in[1,\infty]$, and let $\tau_p$ be the
functional on $D(h_0)$ defined by
\begin{align*}
\tau_p(v) :=
{} & {\RRe\dual{\cA\nabla v}{\nabla v}}
     - (1-\rfrac2p)^2 \dual{\cA_0^\s \nabla |v|}{\nabla |v|} \\
   & - 2|1-\rfrac2p| \int_\Omega \left|
       \inner{\cA_1^\s\nabla |v|}{\IIm (\overline{\sgn v}\,\nabla v)}\right| \\
   & + 2 \dual{|v|\RRe(\tfrac1p b_1- \tfrac1{p'} b_2) }{\nabla |v|}
     - \dual{|v|\IIm(\bplusb)}{\IIm (\overline{\sgn v}\,\nabla v)} + V(v).
\end{align*}
If in the third term of the above definition one removes the absolute value signs
except on $|v|$, then
the corresponding functional coincides with the one introduced in \cite{CM2006}, where
Lemma~1 states that its positivity on $C_\c^1$ is a necessary condition for
formal dissipativity of $-\cL$ in $L^p$.
In Theorem~\ref{mainresult2} we show that a suitable coercivity assumption on $\tau_p$ is a sufficient condition
for the existence of a quasi-contractive $C_0$-semigroup on $L^p$ corresponding to $\cL$;
this sufficient condition is close to the necessary condition in \cite{CM2006}.

Since $\nabla |v| = \RRe (\overline{\sgn v}\,\nabla v)$, it is straightforward to verify that
\begin{equation}\label{tau_p}
\begin{split}
\tau_p(v) =
{} & {\RRe\form(v)}
     - (1-\rfrac2p)^2 \amax(|v|)
     - 2|1-\rfrac2p| \int_\Omega \left|
       \inner{\cA_1^\s\nabla |v|}{\IIm (\overline{\sgn v}\,\nabla v)}\right| \\
   & - (1-\rfrac2p)\dual{|v|\RRe(\bplusb)}{\nabla |v|}
\end{split}
\end{equation}
for all $v\in D(h_0)$, with $\amax$ as defined in~\eqref{amax-def}.

Given $p \in [1,\infty]$, we define $\tomega_p \in (-\infty,\infty]$ by
\begin{equation}\label{omegap-tilde-def}
  \tomega_p := \inf \set{\omega\in\R}
               {\tau_p(v) \ge -\omega\|v\|_2^2 \mbox{ for all } v \in D(h_0)},
\end{equation}
and we set
\begin{equation}\label{J-def}
  J := \set{p\in[1,\infty)}{\tomega_p < \infty} .
\end{equation}
Observe that $\rfrac1p \mapsto \tau_p(v)$ is concave on $[0,1]$ for all $v\in D(h_0)$.
Hence the function $\rfrac1p \mapsto \tomega_p \in (-\infty,\infty]$ is convex on $[0,1]$
and $J$ is an interval.
In the proof of Theorem~\ref{mainresult} we will show that $J\supseteq I$ (cf.~\eqref{I-def}).

\begin{remark}\label{rcomplex1}
If $p\ne2$ then the functional $\tau_p$ is not a quadratic form, i.e.,
the parallelogram law $\tau_p(u+v) + \tau_p(u-v) = 2\tau_p(u) + 2\tau_p(v)$
does not hold. 
One can show this by choosing $u=w^{1+im}$ and $v=\overline u=w^{1-im}$,
where $m>0$ is sufficiently large and $w\in D(h_0)$ is fixed,
with $0\le w\le1$ a.e.\ and $\nabla w \ne 0$.
Note that $u,v$ are multiples of normal contractions of $w$, hence $u,v \in D(h_0)$
as $h_0$ is a Dirichlet form by Assumption~\ref{assump2}.
The terms $\dual{\cA_0^\s \nabla |u\pm v|}{\nabla |u\pm v|}$ from the definition of $\tau_p$
contribute a quadratic term in $m$ on the left hand side of the parallelogram law, 
whilst all the other contributions on either side are sublinear in~$m$.
So even for real-valued coefficients, $\tau_p$ is not quadratic unless $p=2$.

We recall that the case of real-valued coefficients has already been studied in \cite{sv02}.
In that paper, a different definition of $\tau_p$ is used, making it a quadratic form.
It coincides with our definition of $\tau_p$ only for real-valued functions in the form domain.
\end{remark}

For the next two results we will assume in addition to \ref{assump1}~--~\ref{assump4} that
\begin{Rmlist}
\item \label{assump8} there exist $p_0\in J$, $\mu>0$ and $\omega\in\R$ such that
\[
  \tau_{p_0}(v) \ge \mu h_0(v) - \omega\|v\|_2^2
  \smallbds
\]
for all $v \in D(h_0)$.
\end{Rmlist}

The next theorem is our main result on existence of a quasi-contractive
semigroup on $L^p$ corresponding to $\cL$ for all $p\in J$ (see~\eqref{J-def}).

\begin{theorem}\label{mainresult2}
Let Assumptions \ref{assump1}~--~\ref{assump8} be satisfied.
Then $\sinter{J}\ne\varnothing$, $\U \neq \varnothing$
and there exists a consistent family of $C_0$-semigroups
$S_p=(e^{-t\gen_p})_{t\ge0}$ on $L^p$, $p\in J$, such that
\begin{equation}\label{q-c-main2}
\|S_p(t)\|_{p\to p}\le e^{\tomega_pt}
\end{equation}
for all $t \ge 0$ and, for all $p\in\inter J$, the semigroup $S_p$ is analytic and
associated with $(\form, \U)$ in the sense of Definition~\ref{asso}.
\end{theorem}

\begin{remark}\label{main2-rem}
\begin{alplist}
\item \label{main2-rem(a)} 
The above result extends the main result of \cite{sv02}
to the case of complex coefficients. We point out, however, that
the form-boundedness assumption~\ref{assump4} for the lower-order terms
is not needed in \cite{sv02}.
This generality seems not achievable
in the context of complex coefficients because of the absence of a perturbation theory
analogous to the perturbation theory for positive semigroups.

\item \label{mainresult-rem(d)} 
We shall prove in Proposition~\ref{closed} that the form $\form$ is sectorial and closed
if $2 \in \sinter{J}$.
It follows from Remark~\ref{rem1}\ref{rem1(c)} that then $\gen_2$
is the $m$-sectorial operator associated with $\form$.
(Recall that $-\gen_2$ is the generator of $S_2$.)

\item \label{main2-rem(b)} 
In addition to the assumptions of Theorem~\ref{mainresult2}
assume that $\tomega_\infty < \infty$.
Then the semigroup $S_p$ extrapolates to a weak$^*$-continuous
quasi-contractive semigroup $S_\infty$ on $L^\infty$.
Indeed, since $\rfrac1p \mapsto \tomega_p$ is convex,
the interval $J$~is unbounded and $\lim_{p\to\infty}\tomega_p \le \tomega_\infty$.
For all $f\in L^1\cap L^\infty$, $t > 0$ and large $p \in (1,\infty)$
one has $\|S_p(t)f\|_p \le e^{\tomega_pt} \rule{0pt}{2.2ex} \|f\|_p$,
so for $p\to\infty$ one obtains $\|S_p(t)\|_{\infty\to\infty} \le e^{\tomega_\infty t}$.
This generalizes the sufficiency part of \cite[Theorem~4.6]{Ou}
to operators with unbounded coefficients.
\end{alplist}
\end{remark}

In the case of uniformly elliptic and bounded $\cA$ we can extend the interval $J$
of existence of a semigroup associated with $\gen$.

\begin{theorem}\label{thm-extension}
Let Assumptions \ref{assump1}~--~\ref{assump8} be satisfied.
Let $p\in J$, and let $S_p$ be the $C_0$-semigroup on $L^p$ constructed in Theorem~\ref{mainresult2}.
Suppose in addition that $N\ge 3$, that
$\cA_0^s$ is uniformly elliptic and bounded, i.e.,
there exist constants $c_1,c_2\ge 1$ such that
\begin{equation}\label{uell}
  c_1|\xi|^2\le \<\cA_0^s\xi, \xi\> \le c_2|\xi|^2 \quad \text{for all}\ \xi\in\R^N,
\end{equation}
and that $D(h_0)$ is a subset of $L^{\frac{2N}{N-2}}(\Omega)$ and an ideal of $D(\amax)$;
cf.~\eqref{D-amax-def}.

Let $p_- := \inf J$, $p_+ := \sup J$, $p_{\max} := \frac{N}{N-2}p_+$ and
$p_{\min} := \bigl(\frac{N}{N-2}(p_-)'\bigr)'$.
Then $S_p$ extrapolates to an analytic $C_0$-semigroup on $L^q$
for all $q\in (p_{\min}, p_{\max})$, and the sector of analyticity and the spectrum of the generators are independent of $q$.
\end{theorem}

\begin{remark}\label{remcomplex1.11}
\begin{alplist}
\item \label{remcomplex1.11(a)} 
The assumption of the theorem that $D(h_0)$ is a subset of
\smash{$L^{\frac{2N}{N-2}}(\Omega)$} is in fact a Sobolev embedding theorem
which holds, for example, for Dirichlet boundary conditions or if the domain
$\Omega$ satisfies the cone property or the extension property; see \cite{Ad}.

\item \label{remcomplex1.11(c)} 
In \cite{lsv02}, instead of $D(h_0)$ being an ideal of $D(\amax)$, a slightly weaker
(but more technical) assumption was used, namely that all the bounded Lip\-schitz
functions on $\Omega$ are multiplication operators on $D(h_0)$.
In the case that the bounded Lipschitz functions in $D(h_0)$ form a
core for~$h_0$, the two conditions are in fact equivalent.
\end{alplist}
\end{remark}

The paper is organized as follows.
In Section~\ref{sec_approx} we deal with the first part of Definition~\ref{asso}:
we investigate extrapolation of the semigroup $S_U$ to the $L^p$-spaces.
Next in Section~\ref{Scomplex2} we prove our main result on generation of
quasi-contractive $C_0$-semigroups, Theorem~\ref{mainresult2},
from which we then derive Theorem~\ref{mainresult} in Section~\ref{S-mainresult}.
There we also discuss some examples.
Finally, in Section~\ref{extension} we prove Theorem~\ref{thm-extension}.

\section{Construction of approximating semigroups}\label{sec_approx}

In this section we study the $C_0$-semigroup $S_U$ on $L^2$ associated with
the form $\form+U$, where $\mkU \in \U$. 
Using Assumption~\ref{assump8} we will show in Proposition~\ref{appr} that $S_U$
extrapolates to a family of consistent $C_0$-semigroups $S_\Up$ on $L^p$, $p \in J$,
with an estimate of the growth bound independent of $\mkU$.
If $p \in \inter J$, then we also obtain a common sector of analyticity of the semigroups $S_\Up\mkern1mu$.
\pagebreak[1]

We first show that Assumption~\ref{assump8} extends to all $p \in \sinter{J}$
(see~\eqref{J-def} for the definition of $J$).

\begin{lemma} \label{lcomplex204}
Let Assumptions~\ref{assump1}~--~\ref{assump8} be satisfied, and let $p\in[1,\infty)$.
Then $p \in \sinter{J}$ if and only if there exist $\mu_p>0$ and
$\omega_p\in\R$ such that
\begin{equation}\label{mu-omega}
\tau_p(v) \ge \mu_p h_0(v) - \omega_p\|v\|_2^2
\end{equation}
for all $v \in D(h_0)$.
In particular, $\sinter{J} \neq \varnothing$.
\end{lemma}
\begin{proof}
Assumptions \ref{assump1}~--~\ref{assump4} imply that there
is a constant $C>0$ such that
\begin{equation}\label{tau-lip}
|\tau_p(v)-\tau_q(v)| \le C\left|\frac1p-\frac1q\right| (h_0+1)(v)
\avoidbreak
\end{equation}
for all $v \in D(h_0)$ and $p,q \in [1,\infty]$.

Now let $p\in[1,\infty)$ satisfy~\eqref{mu-omega} for some $\mu_p>0$ and
$\omega_p\in\R$. Assume that $p=1$.
Set $U = \inner{(\cA_0^\s)^{-1}\RRe b_1}{\RRe b_1} + V^+$;
then an easy computation shows that
$\tau_1(v) = 2\dual{(\RRe b_1)v}{\nabla v} + V(v)
         \le \frac{\mu_1}{2} h_0(v) + \rfrac{2}{\mu_1} U(v)$
for all real-valued $v\in D(h_0)$.
From this estimate and~\eqref{mu-omega} one deduces that
$\frac{\mu_1}{2} h_0 \le \rfrac{2}{\mu_1} \mkU + \omega_1$ on $D(h_0)$.
By Lemma~\ref{dense} below (applied with $L=D(h_0)$) it follows that
$\Q(U) \subseteq D(h_0) \subseteq \smash{W^{1,1}_\loc}$, which is impossible:
For a given $x\in\Omega$ and large enough $n\in\N$,
the indicator function $\ind_{B(x,1/n)\cap[U\le n]}$
lies in $\Q(U)$ but not in $W^{1,1}_\loc\mkern-1mu$.
Thus we have shown $p>1$.
Now~it follows from \eqref{mu-omega} and~\eqref{tau-lip} that $p \in \sinter{J}$.
In particular, we have shown that $p_0 \in \sinter{J}$.

Conversely, \eqref{mu-omega} is valid for $p=p_0$ by Assumption~\ref{assump8}.
Moreover, $\tau_q(v) \ge - \tomega_q \|v\|_2^2$ for all $q\in J$
and $v\in D(h_0)$ by~\eqref{omegap-tilde-def}.
Then by the concavity of $\rfrac1p \mapsto \tau_p(v)$
we obtain \eqref{mu-omega} for all $p \in \inter{J}$.
\end{proof}

\begin{lemma}\label{dense}
Let $(X,\mu)$ be a measure space, let $L$ be a dense sublattice of $L^2(\mu)$,
and let $U\from X\to[0,\infty)$ be a measurable function satisfying $L \subset \Q(U)$.
Then $L$ is dense in $\Q(U)$.
\end{lemma}
\begin{proof}
It suffices to show that the closure of $L$ in $\Q(U)$ contains every
function $0 \le f \in \Q(U)$.
Let $(f_n)$ be a sequence in $L$ converging to $f$ in $L^2(\mu)$,
without loss of generality $0 \le f_n \to f$ a.e.
Then for every $k\in\N$ we obtain
\[
  L \ni f_n\wedge f_k \to f\wedge f_k
\]
in $\Q(U)$ as $n\to\infty$, by the dominated convergence theorem.
Moreover, $f\wedge f_k \to f$ as $k\to\infty$ in $\Q(U)$.
It follows that $f$ lies in the closure of $L$ in $\Q(U)$.
\end{proof}

In the following we fix some $\mu_p>0$ and $\omega_p\in\R$ as in Lemma~\ref{lcomplex204},
for every $p \in \sinter{J}$.
The next result is a simple consequence of Lemma~\ref{lcomplex204}.

\begin{proposition}\label{closed}
Let Assumptions~\ref{assump1}~--~\ref{assump8} be satisfied,
and suppose that $2\in\sinter{J}$.
Then $\form$ is a closed sectorial form in $L^2$.
\end{proposition}
\begin{proof}
It follows from~\eqref{tau_p} that $\RRe \form(v) = \tau_2(v)$ for all $v \in D(h_0)$.
Thus, by Lemma~\ref{lcomplex204} there exist $\mu_2>0$ and $\omega_2\in\R$ such that
$\RRe \form(v) \ge \mu_2 h_0(v) - \omega_2\|v\|_2^2$ for all $v \in D(h_0)$.
This implies the assertion since $\form$ is a bounded form on $D(h_0)$.
\end{proof}

We define the potential $\Uhat \colon \Omega \to [0,\infty)$ by 
\begin{equation}\label{Uhat}
  \Uhat := \frac14 \+ \inner{(\cA_0^s)^{-1} \RRe(\bplusb)}{\RRe(\bplusb)}.
\end{equation}
This potential will play an important role in the proof of Theorem~\ref{mainresult2}.

\begin{lemma} \label{lcomplex202}
Let Assumptions~\ref{assump1}~--~\ref{assump8} be satisfied.
\begin{alplist}[i]
\item \label{lcomplex202-1}
Then $\mkUhat \in \U$ and $\form+\Uhat$ is closed.
In particular, $\U\ne\varnothing$.
\item \label{lcomplex202-2}
If $\mkU \in \U$ satisfies $U \ge \mkUhat - c_1$ for some $c_1\in\R$, then $\form+U$ is closed.
\end{alplist}
\end{lemma}
\begin{proof}
\ref{lcomplex202-1}
One easily sees that $\mkUhat \le \frac c2 (h_0+1)$,
where $c > 0$ is as in Assumption~\ref{assump4}.
It follows from~\eqref{tau_p} that
\begin{align}
\RRe\form(v)
 &\ge \tau_p(v) + \bigl(1-\rfrac2p\bigr) \dual{|v|\RRe(\bplusb)}{\nabla|v|}
                + (1-\rfrac2p)^2 \amax(|v|) \nonumber \\
 &\ge \tau_p(v) - \tfrac14 \dual{|v|(\cA_0^s)^{-1} \RRe(\bplusb)}{|v|\RRe(\bplusb)}
      \alignstrut \nonumber \\
 &= \tau_p(v) - \Uhat(v) \alignstrut \label{t+Uhat}
\end{align}
for all $v\in D(h_0)$ and $p \in [1,\infty)$. 
Thus, Assumption~\ref{assump8} implies that $\form+\Uhat$ is a closed sectorial form.
In particular we obtain $\mkUhat \in \U$.

\pagebreak[1]

\ref{lcomplex202-2}
By the assumptions on $U$ there exists a $C>0$ such that
$\RRe\form+\mkUhat - c_1 \le \RRe\form + \mkU \le C(h_0 + 1)$.
Hence the assertion follows from \ref{lcomplex202-1}.
\end{proof}

In the following lemma we provide some estimates on the form $\form$
that are needed for the proof of Proposition~\ref{appr}, the main result of this section.
Here we adopt the convention $0^s = \infty$ for all $s<0$,
so that $0^s\wedge r = r$ for all $r\ge0$.

\begin{lemma}\label{accr}
Let Assumptions~\ref{assump1}~--~\ref{assump4} be satisfied.
Let $u \in D(h_0)$, $r>1$ and $p \in (1,\infty)$.
Set $v = u\bigl(|u|^{\frac p2 -1}\wedge r\bigr)$,
$w = u\bigl(|u|^{p-2}\wedge r^2\bigr)$ and $\smash{ \chi = \chiind{\ge r}}$.
Then $v,w \in D(h_0)$ and
\begin{equation}\label{form-bound-below}
\RRe \form(u,w) \ge \tau_p(v) - (\chi\Uhat)(v).
\end{equation}
Moreover, there exists a $c_0>0$, depending only on the
constants in Assumptions~\ref{assump3} and~\ref{assump4}, such
that
\begin{equation}\label{im-est}
  \bigl|\IIm \form(u,w)\bigr| \le c_0(h_0+1)(v).
\end{equation}
\end{lemma}

\begin{proof}
For any $s\in\R$ the function $\C \ni z \mapsto z(|z|^s\wedge r)$ is Lipschitz continuous;
this is an easy consequence of the Lipschitz continuity of $0\le x \mapsto x^{s+1}\wedge(rx)$.
Since $h_0$ is a Dirichlet form, it follows that $v,w\in D(h_0)$.

We define the complementary indicator function
$\chic := \ind_\Omega-\chi = \smash{\chiind{<r}}$ and observe that
\begin{equation}\label{uw}
  u = v\bigl(\rfrac1r \chi + \chic |v|^{\twop-1}\bigr)
  \quad \mbox{and} \quad
  w = v\bigl(r\chi + \chic |v|^{1-\twop}\bigr).
\end{equation}
Moreover, by \cite[Lemma~5.2]{sv02} we obtain
\begin{align*}
\nabla v
 &= r\chi \nabla u + \chic |u|^{\frac p2-1}\bigl( \nabla u + (\tfrac p2-1)(\sgn u)\nabla |u| \bigr), \\
\nabla w
 & = r^2\chi \nabla u + \chic |u|^{p-2}\bigl( \nabla u + (p-2)(\sgn u)\nabla |u| \bigr) \alignstrut
\end{align*}
and $\nabla|v| = (r\chi + \tfrac p2\chic |u|^{\frac p2-1})\nabla |u|$.
The latter implies that $\chic (1-\rfrac2p) \nabla |v| =
\chic (\frac p2-1) |u|^{\frac p2-1} \nabla |u|$.
Taking into account $\chi_\c|v| = \chi_\c|u|^\frac{p}{2}$,
we deduce that
\begin{align*}
\nabla u &= \rfrac1r \chi \nabla v + \chic |v|^{\twop-1}
            \bigl( \nabla v - (1-\rfrac2p) (\sgn v) \nabla |v| \bigr), \\
\nabla w &= r \chi \nabla v + \chic |v|^{1-\twop} \bigl( \nabla v + (1-\rfrac2p) (\sgn v) \nabla |v| \bigr), \alignstrut
\end{align*}
and denoting
\[
  \zeta = \overline{\sgn v} \, \nabla v, \quad
  \xi = \RRe\zeta = \nabla |v|, \quad
  \eta = \IIm \zeta,
  \smallbds
\]
we arrive at
\begin{equation}\label{grad-uw}
\begin{split}
\overline{\sgn v}\,\nabla u &= \rfrac1r \chi\zeta
  + \chic |v|^{\twop-1} \bigl( \zeta - (1-\rfrac2p) \xi \bigr), \\
\overline{\sgn v}\,\nabla w &= r \chi\zeta
  + \chic |v|^{1-\twop} \bigl( \zeta + (1-\rfrac2p) \xi \bigr). \alignstrut
\end{split}
\end{equation}
It follows that
\begin{align*}
\<\cA\nabla u, \nabla w\>
 &= \< \cA\,\overline{\sgn v}\,\nabla u, \overline{\sgn v}\,\nabla w \> \\[0.5ex]
 &= \<\cA\zeta, \zeta\>
    + (1-\rfrac2p) \chic \bigl( \<\cA\zeta, \xi\> - \<\cA\xi, \zeta\> \bigr)
    - (1-\rfrac2p)^2 \chic \<\cA\xi, \xi\>.
\end{align*}
Using the identity
\begin{align*}
\<\cA\zeta, \xi\> - \<\cA\xi, \zeta\>
 &= \<\cA i\eta, \xi\> - \<\cA\xi, i\eta\> \\[0.5ex]
 &= - \<(\cA+\cA^\top)\xi, i\eta\>
  = 2i\<\cA_0^\s\xi, \eta\> - 2\<\cA_1^\s\xi, \eta\>
\end{align*}
(recall $\cA = \cA_0 + i\cA_1$),
we thus infer that
\begin{equation}\label{ReA}
\RRe \<\cA\nabla u, \nabla w\>
  = \RRe \<\cA\zeta, \zeta\>
    - 2(1-\rfrac2p) \chic \<\cA_1^\s\xi, \eta\>
    - (1-\rfrac2p)^2 \chic \<\cA_0^\s\xi, \xi\>
\end{equation}
and
\begin{equation}\label{ImA}
\IIm \<\cA\nabla u, \nabla w\>
  = \IIm \<\cA\zeta, \zeta\>
    + 2(1-\rfrac2p) \chic \<\cA_0^\s\xi, \eta\>
    - (1-\rfrac2p)^2 \chic \<\cA_1^\s\xi, \xi\> \rlap.
\end{equation}

By \eqref{uw} and~\eqref{grad-uw} we have
\begin{align*}
\overline w \nabla u
 &= |v| \bigl( \zeta - (1-\rfrac2p) \chic \xi \bigr)
  = |v| \bigl( \xi+i\eta - (1-\rfrac2p) (\ind_\Omega-\chi) \xi \bigr) \\[0.8ex]
 &= |v| \bigl( \rfrac2p \xi + (1-\rfrac2p) \chi \xi + i\eta \bigr) \\[-1.3\baselineskip]
\end{align*}
and
\begin{align*}
u \nabla \overline w
 &= |v| \bigl( \+ \overline \zeta + (1-\rfrac2p) \chic \xi \bigr)
  = |v| \bigl( \xi-i\eta + (1-\rfrac2p) (\ind_\Omega-\chi) \xi \bigr) \\[0.8ex]
 &= |v| \bigl( \tfrac2{p'} \xi - (1-\rfrac2p) \chi \xi - i\eta \bigr) ,
\end{align*}
so it follows that
\begin{equation}\label{b1b2}
\begin{split}
\MoveEqLeft
\<b_1\cdot\nabla u, w\> - \< b_2\+u, \nabla w\> \\[0.5ex]
 &= |v| \bigl( \<\tfrac2p b_1 - \tfrac2{p'} b_2,\xi\>
          + (1-\rfrac2p) \chi \<\bplusb,\xi\> + i \<\bplusb,\eta\> \bigr) .
\end{split}
\end{equation}

Now we are ready to estimate $\RRe \form(u,w)$. Using the
definition~\eqref{Uhat} of~$\Uhat$ we obtain $(1-\rfrac2p)
\bigl|\<\RRe(\bplusb), \xi\>\bigr| \le (1-\rfrac2p)^2
\<\cA_0^\s\xi, \xi\> + \Uhat|v|^2$ and hence
\begin{align*}
\MoveEqLeft
\RRe \bigl( \<b_1\cdot\nabla u, w\> - \< b_2\+u, \nabla w\> \bigr) \\
 &\ge |v| \inner{\RRe(\tfrac2p b_1 - \tfrac2{p'} b_2)}{\xi}
     - (1-\rfrac2p)^2 \chi \<\cA_0^\s\xi, \xi\>
     - \chi \Uhat|v|^2 - |v| \<\IIm(\bplusb),\eta\> .
\end{align*}
Together with~\eqref{ReA} and the identity $(\RRe Q)(u,w) = V(v)$, we conclude
by the definition~\eqref{form-def} of $\form$ that
\begin{align*}
\RRe \form(u,w)
\ge {} & \smash[t]{ \RRe \dual{\cA \zeta}{\zeta}
          - 2|1-\rfrac2p| \int_{\Omega} \bigl|\<\cA_1^\s\xi, \eta\>\bigr|
          - (1-\rfrac2p)^2 \dual{\cA_0^\s\xi}{\xi}
          + V(v) } \\
       & {} + \dual{|v|\RRe(\tfrac2p b_1- \tfrac2{p'} b_2)}{\xi}
         - \dual{|v|\IIm(\bplusb)}{\eta} - (\chi\Uhat)(v) \\[0.8ex]
  = {} & \tau_p(v) - (\chi \Uhat)(v).
\end{align*}
This completes the proof of~\eqref{form-bound-below}.

\pagebreak[1]

Now we estimate $\bigl|\IIm \form(u,w)\bigr|$. It follows
from~\eqref{ImA} and Assumption~\ref{assump3} that
\begin{align*}
\bigl|\IIm \<\cA\nabla u, \nabla w \>\bigr|
 &= \bigl|\IIm\<\cA\zeta, \zeta \>
    + 2(1-\rfrac2p) \chic \<\cA_0^\s\xi, \eta\>
    - (1-\rfrac2p)^2 \chic \<\cA_1^\s\xi, \xi\>\bigr| \\
 &\le 3c \bigl( \<\cA_0^\s\xi,\xi\> + \<\cA_0^\s\eta,\eta\> \bigr)
    = 3c \<\cA_0^\s\nabla v, \nabla v\>.
\end{align*}
Next, with $U_j := \<(\cA_0^\s)^{-1} b_j, b_j\>$ for
$j=1,2$, we infer from~\eqref{b1b2} that
\begin{align*}
\bigl| \IIm \bigl( \<b_1\cdot\nabla u, w\> &- \< b_2\+u, \nabla w\> \bigr) \bigr| \\[0.5ex]
 &\le |v| \bigl( 2|\<\IIm b_1, \xi\>| + 2|\<\IIm b_2, \xi\>| + |\<\RRe(\bplusb),\eta\>| \bigr) \\[0.5ex]
 &\le 2\<\cA_0^\s\xi, \xi\> + 2 \<\cA_0^\s\eta, \eta\> + \tfrac14(U_1+U_2) |v|^2.
\end{align*}
Finally $\IIm Q(u,w) = \IIm Q(v)$. Hence \eqref{im-est} follows from Assumption~\ref{assump4}.
\end{proof}

In the following let $\h_p$ denote the lower semi-continuous hull of $\tau_p$,
for given $p\in J$; in other words, the functional $\h_p \colon L^2 \to (-\infty,\infty]$ 
is defined by
\[
  \h_p(v) := \sup\set{\h(v)}{\h\ \text{is lower semi-continuous on}\ L^2\
                             \text{and}\ \h \le \tau_p\ \text{on}\ D(h_0)}.
\]
By~\eqref{mu-omega} and \cite[Lemma~VIII.3.14a]{kato80} we have
\begin{equation}\label{mu-omega-used}
\h_p(v) \ge 
\begin{cases}
\mu_p h_0(v) - \omega_p\|v\|_2^2 & \text{if } v \in D(h_0), \\[0.4ex]
                          \infty & \text{if } v \in L^2 \setminus D(h_0)
\end{cases}
\end{equation}
for all $p\in\sinter J$. Similarly, 
\begin{equation}\label{omegap-tilde-def-used}
\h_p(v) \ge - \tomega_p \|v\|_2^2
\end{equation}
for all $v \in D(h_0)$ and
$p\in J$ by \eqref{omegap-tilde-def}.
If $\cA_1 = 0$, then it is not hard to show that $\tau_p$ is lower semi-continuous
for all $p \in \sinter J$, so $\h_p|_{D(h_0)} = \tau_p$ in that case.

For the next result recall that 
$S_U$ is the $C_0$-semigroup on $L^2$ associated with the closure of
$\form+U$, for given $\mkU\in\U$.

\begin{proposition}\label{appr}
Let Assumptions~\ref{assump1}~--~\ref{assump8} be satisfied.
Let $\mkU \in \U$ and $p\in J$.
Then the semigroup $S_U$ extrapolates to a $C_0$-semigroup $S_\Up$ on $L^p$,
and $\|S_\Up(t)\|_{p\to p}\le e^{\tomega_pt}$ for all $t\ge0$.

Let $-\gen_\Up$ be the generator of $S_\Up\mkern1mu$.
If $p\in\sinter{J}$, then for all $u\in D(\gen_\Up)$ we have
\begin{equation}\label{lconc}
  \RRe \dual{\gen_\Up u}{u|u|^{p-2}} \ge \h_p(u|u|^{\frac p2-1})
  \smallbds
\end{equation}
and
\begin{equation}\label{lanal}
\mkern-10mu  
   \left|\IIm\dual{\gen_\Up u}{u|u|^{p-2}}\right|
   \le \rfrac{c_0}{\mu_p}\RRe\dual{(\gen_\Up+\omega_p+\mu_p)u}{u|u|^{p-2}},
\mkern-10mu
\end{equation}
with $\mu_p>0$ and $\omega_p\in\R$ as in~\eqref{mu-omega}
and $c_0$ as in~\eqref{im-est}.
In particular, $\gen_\Up$ is an $m$-sectorial operator of
angle $\arctan\rfrac{c_0}{\mu_p}$ and $S_\Up$ is an analytic
semigroup on $L^p$.
\end{proposition}

Although the above proposition is similar to \cite[Lemma~5.1]{sv02},
we provide a self-contained proof for the reader's convenience.
We will use the following two lemmas.

\begin{lemma}\label{srs}
Let $p \in (1,\infty)$, and let $(A_k)_{k \in \N}$ be a sequence of
closed operators in $L^p$ that converges in the strong resolvent sense
to a closed operator $A_\infty$ in~$L^p$.
\begin{alplist}[i]
\item \label{srs(a)} 
Let $\h \from L^2 \to \R\cup\{\infty\}$ be a lower semi-continuous functional,
and assume that
\begin{equation}\label{vp-est}
  \RRe\dual{A_k u}{u|u|^{p-2}} \ge \h\bigl(u|u|^{\frac p2 -1}\bigr)
\end{equation}
for all $k\in \N$ and $u\in D(A_k)$.
Then~\eqref{vp-est} also holds for $k=\infty$ and all $u \in D(A_\infty)$.

\item \label{srs(b)} 
Assume that there exist $C\ge0$ and $\omega\in\R$ such that
\begin{equation}\label{sect-est}
  \bigl|\IIm\dual{A_k u}{u|u|^{p-2}}\bigr| \le C \RRe\dual{(\omega+A_k)u}{u|u|^{p-2}}
\end{equation}
for all $k\in \N$ and $u\in D(A_k)$.
Then~\eqref{sect-est} also holds for $k=\infty$ and all $u \in D(A_\infty)$.

\item \label{srs(c)} 
Let $k \in \N$.
In both \eqref{vp-est} and~\eqref{sect-est},
the estimate holds for all $u\in D(A_k)$ if it is satisfied on a core for $A_k$.
\end{alplist}
\end{lemma}
\begin{proof}
Given $u \in L^p$, we denote $v_p(u) := u|u|^{\frac p2 -1}$ and $w_p(u) := u|u|^{p-2}$.

\ref{srs(a)} Let $\lambda \in \C$ be such that $(\lambda+A_k)^{-1} \to (\lambda+A_\infty)^{-1}$
strongly. Let $u\in D(A_\infty)$ and set $u_k=(\lambda + A_k)^{-1}(\lambda+A_\infty)u$
for all $k \in \N$.
By~\eqref{vp-est} we have
\[
\RRe\dual{\lambda u_k}{w_p(u_k)} + \h\bigl(v_p(u_k)\bigr)
 \le \RRe\dual{(\lambda + A_k)u_k}{w_p(u_k)}
 =   \RRe\dual{(\lambda + A_\infty)u}{w_p(u_k)}
\]
for all $k \in \N$. Moreover, $u_k \to u$ in $L^p$, so $v_p(u_k) \to v_p(u)$ in
$L^2$ and $w_p(u_k) \to w_p(u)$ in $L^{p'}$ as $k\to\infty$.
Since $\h$ is lower semi-continuous, we conclude that
$\h(v_p(u)) \le \RRe\dual{A_\infty u}{w_p(u)}$.

\ref{srs(b)} This is proved in a similar way.

\ref{srs(c)} Let $u\in D(A_k)$, and let $(u_m)$ be a sequence from the core such that
$u_m \to u$ in $D(A_k)$. Then $v_p(u_m) \to v_p(u)$ in $L^2$
and $w_p(u_m) \to w_p(u)$ in $L^{p'}$ as $m\to\infty$,
and the assertion follows (use the lower semi-continuity of $\h$ for~\eqref{vp-est}).
\end{proof}

\begin{lemma}\label{extrapol}
Let $q \in [1,\infty)$, and let $T$ be a $C_0$-semigroup on $L^q$ with generator $-A$.
Let $p \in (1,\infty)$ and $\omega\in\R$. 
Assume that
for each $u\in D(A)$ there exists a sequence $(w_n)$ in $L^{p'} \cap L^{q'}$
such that 
$|w_n| \le |u|^{p-1}$, $u\olwn \ge 0$ and
\[
  \RRe\dual{(\omega+A)u}{w_n} \ge 0
\]
for all $n\in\N$, and $|w_n| \to |u|^{p-1}$ a.e.
Then $T$ extrapolates to a quasi-contractive $C_0$-semigroup on $L^p$.
\end{lemma}
\begin{proof}
Without loss of generality assume that $\omega=0$.
Let $f \in L^p \cap L^q$ and $0 < \lambda \in \rho(-A)$.
We shall show that $\|(\lambda+A)^{-1}f\|_p \le \frac1\lambda \|f\|_p$.
Set $u = (\lambda+A)^{-1}f$, and let $(w_n)$ be a sequence as in the assumption.
Then $|w_n|^{p'-1} = |w_n|^{1/(p-1)} \le |u|$ for all $n\in\N$, so we obtain
$|w_n|^{p'} \le |uw_n| = u\olwn$ and hence
\[
  \lambda\|w_n\|_{p'}^{p'}
  \le \lambda \int u\olwn\,
  \le \RRe\dual{(\lambda+A)u}{w_n}
  \le \|f\|_p \|w_n\|_{p'}.
\]
It follows that $\bigl\||w_n|^{1/(p-1)}\bigr\|_p = \|w_n\|_{p'}^{p'-1}
\le \frac1\lambda \|f\|_p$. Since $|w_n|^{1/(p-1)} \to |u|$ a.e., Fatou's lemma yields
$\|(\lambda+A)^{-1}f\|_p = \|u\|_p \le  \frac1\lambda \|f\|_p$.
By the exponential formula we conclude that $T$ extrapolates to
a contractive semigroup on $L^p$, which is strongly continuous
by \cite[Proposition~1]{voi92}.
(The strong continuity can also be deduced from \cite[Theorem~1]{BL}.)
\end{proof}

For the proof of Proposition~\ref{appr} we need in addition the next observation.

\begin{observation}\label{lsc}
Let $M$ be a metric space, and for all $n\in\N$ let $f_n \from M \to \R\cup\{\infty\}$
be lower semi-continuous. Assume that $f_n \uparrow f$ pointwise,
and let $(x_n)$ be a convergent sequence in $M$.
Then $f(\lim x_n) \le \liminf f_n(x_n)$.
This holds since for any $m\in\N$ one can estimate
$f_m(\lim x_n) \le \liminf\limits_{n\to\infty} f_m(x_n) \le \liminf\limits_{n\to\infty} f_n(x_n)$.
\end{observation}

\begin{proof}[Proof of Proposition~\ref{appr}]
First assume that $p>1$ and that $U \ge \mkUhat-c$ for some $c\ge0$.
Then $\form+U$ is a closed sectorial form by Lemma~\ref{lcomplex202}\ref{lcomplex202-2}.
Let $u\in D(\gen_U)$. Then $u \in D{(\form+U)} = D(h_0)$.
Let $n\in\N$, and set $v_n = u\bigl(|u|^{\frac p2 -1}\wedge n\bigr)$,
$w_n = u\bigl(|u|^{p-2}\wedge n^2\bigr)$ and $U_n = \chiind{\ge n}\Uhat$.
Then $|v_n|^2 = u\olwn$ and $\mkU-U_n \ge -c$,
so by Lemma~\ref{accr} and~\eqref{omegap-tilde-def} we obtain
\[
  \RRe\dual{\gen_U u}{w_n}
  = \RRe (\form +U)(u,w_n)
  \ge \tau_p(v_n) + (\mkU-U_n)(v_n)
  \ge -(\tomega_p+c) \|v_n\|_2^2\,.
\]
Thus, by Lemma~\ref{extrapol}, $S_U$ extrapolates to a $C_0$-semigroup $S_\Up$ on~$L^p$.

Let $-\gen_\Up$ denote the generator of $S_\Up$.
Let $u\in D(\gen_U)\cap D(\gen_\Up)$, and let $v_n$, $w_n$ and $U_n$ be as above.
Then $v_n \to u|u|^{\frac{p}{2}-1}\rule{0pt}{2.3ex}$ in $L^2$
and $w_n \to u|u|^{p-2}$ in $L^{p'}$ and hence
\[
  \lim_{n\to\infty} (\form+U)(u,w_n)
  = \dual{\gen_U u}{u|u|^{p-2}}
  = \dual{\gen_\Up u}{u|u|^{p-2}}.
\]
For each $n\in\N$, the functional $\h_p + (\mkU-U_n)$ on~$L^2$ given by
$v \mapsto \h_p(v) + \int (\mkU-\nobreak U_n)|v|^2$ is lower semi-continuous.
Hence it follows from Lemma~\ref{accr} and Observation~\ref{lsc} that
\begin{align*}
\RRe\dual{\gen_\Up u}{u|u|^{p-2}}
 &= \lim_{n\to\infty} \RRe(\form+U)(u,w_n)
 \ge \liminf_{n\to\infty} (\tau_p+\mkU-U_n)(v_n)  \\
 &\ge \liminf_{n\to\infty} (\h_p+\mkU-U_n)(v_n)
  \ge (\h_p+U)(u|u|^{\frac{p}{2}-1}),
\end{align*}
where $\h_p+U$ is considered as a functional on $L^2$.
The set $D(\gen_U)\cap D(\gen_\Up)$ is a core for $\gen_\Up$,
so by Lemma~\ref{srs}\ref{srs(c)} and~\eqref{omegap-tilde-def-used} we conclude that
\begin{equation}\label{hp+U}
  \RRe\dual{\gen_\Up u}{u|u|^{p-2}}
  \ge (\h_p+U)(u|u|^{\frac{p}{2}-1})
  \ge -\tomega_p\|u\|_p^p
  \avoidbreak
\end{equation}
for all $u \in D(\gen_\Up)$.
Thus~\eqref{lconc} holds (even for all $p\in J\setminus\{1\}$),
and by the Lumer--Phillips theorem it follows that
$\|S_\Up(t)\|_{p\to p}\le e^{\tomega_pt}$ for all $t\ge0$.

In the case $p=1$ we can argue as in Remark~\ref{main2-rem}\ref{main2-rem(b)}:
we have $\lim_{p\to1}\tomega_p \le \tomega_1$, so by the above we obtain a semigroup
$S_{U,1}$ on $L^1$ with $\|S_{U,1}(t)\|_{1\to 1}\le e^{\tomega_1t}$ for all $t\ge0$.
Moreover, $S_{U,1}$ is strongly continuous by \cite[Proposition~4]{voi92}.

Now assume that $p\in\sinter J$.
Let $u\in D(\gen_U)\cap D(\gen_\Up)$, and let again $v_n$, $w_n$ and $U_n$ be as above.
Since $u\olwn$ is real, we obtain
\[
  \IIm\dual{\gen_\Up u}{w_n}
  = \IIm(\form+U)(u,w_n) = \IIm\form(u,w_n)
\]
for all $n \in \N$. Set $M = \rfrac{c_0}{\mu_p}$ and $\omega = \mu_p + \omega_p$.
Using \eqref{im-est}, \eqref{mu-omega}, \eqref{form-bound-below} and $U\ge0$, we estimate
\begin{align*}
\mathopen|\IIm\form(u,w_n)|
 &\le c_0(h_0+1)(v_n)
  \le M(\tau_p+\omega)(v_n) \\[0.5ex]
 &\le M(\RRe\form+\mkU+\omega)(u,w_n) + MU_n(v_n).
\end{align*}
Next, $U_n|v_n|^2 \le \Uhat|u|^p
\le (\mkU+c) \bigl| u|u|^{\frac{p}{2}-1} \bigr|\rule{0pt}{1.7ex}^2 \in L^1$
by~\eqref{hp+U} and $U_n|v_n|^2 \to 0$ a.e.
Therefore $U_n(v_n) \to 0$ by the dominated convergence theorem, and we infer that
\[
  \mathopen|\IIm \dual{\gen_\Up u}{u|u|^{p-2}}|
  = \lim_{n\to\infty} \mathopen|\IIm\form(u,w_n)|
  \le M\RRe\dual{(\gen_\Up+\omega) u}{u|u|^{p-2}}.
\]
By Lemma~\ref{srs}\ref{srs(c)}, this estimate carries over to all $u\in D(\gen_\Up)$,
i.e., \eqref{lanal} holds.

So far we have proved the proposition in the case where
$U\ge\mkUhat-c$ for some $c\ge0$.
In this last step we show the assertions for an arbitrary $\mkU\in\U$.
Let $k\in\N$ and set $\Utilde_k = \mkU+(\mkUhat-k)^+$.
Then $\Utilde_k \ge \mkUhat-k$,
and $\Utilde_k \in \U$ since $\Utilde_k$ is $h_0$-form-bounded and
$(\form+U)+(\mkUhat-k)^+$ is closable as a sum of two closable forms.
Thus, as shown above, the assertions hold for $\Utilde_k$ in place of~$\mkU$.

Note that $(\form+\Utilde_k)(v) \to (\form+U)(v)$ for all $v\in D(h_0)$
due to the dominated convergence theorem; then
by \cite[Theo\-rem~VIII.3.6]{kato80} we see that $S_{\smash[t]{\Utilde_k}}\to S_U$
as $k\to\infty$ in the semigroup sense on~$L^2$.
Now let $p \in J$. Using Fatou's lemma, we obtain
$\|S_U(t)\|_{p\to p} \le e^{\tomega_p t}$ for all $t \ge 0$.
Then by \cite{voi92} one deduces that $S_U$
extrapolates to a $C_0$-semigroup $S_\Up$ on~$L^p$.
If $p\in\sinter J$ then it follows by interpolation that
$S_{\smash[t]{\Utilde_k}\shift,p} \to S_\Up$ as $k\to\infty$ in the semigroup sense on $L^p$.
Applying Lemma~\ref{srs} we thus conclude that \eqref{lconc} and~\eqref{lanal} hold.
\end{proof}

With the same argument as in the last paragraph of the above proof
one also obtains the following result.

\begin{lemma}\label{lp-conv}
Suppose that Assumptions~\ref{assump1}~--~\ref{assump8} are satisfied,
and let $p\in\sinter{J}$.
Let $U,\Utilde,U_1,U_2,\ldots \in \U$ satisfy $\mkU \le U_k \le \Utilde$ for all $k\in\N$
and $U_k\to U$ a.e.\ as $k\to\infty$.
Then $S_{U_k\shift,p} \to S_\Up$ as $k\to\infty$ in the semigroup sense on $L^p$.
\end{lemma}

We end the section by commenting on the reference form $h_0$.

\begin{remark}\label{reference-form}
Recall from Assumption~\ref{assump2} that $h_0 \subset h_{\max}$ is a Dirichlet form.
We point out that in \cite{sv02},
where the case of real-valued coefficients is studied,
the main reference object is a Dirichlet form $a_0\subset \amax$ rather than $h_0$,
and it is assumed that $D(a_0) \cap \Q(V^+)$ is a core for~$a_0$ (recall $V=\RRe Q$).
Under that assumption one can choose $h_0 := a_0 + V^+$, and then
$h_0$ satisfies Assumption~\ref{assump2}.

Conversely, if in the setting of the current paper one defines
$a_0=\overline{\amax\restrict{D(h_0)}}$,
then $a_0$ is a Dirichlet form, $a_0 \subset \amax$, and
$D(a_0) \cap \Q(V^+)$ is a core for~$a_0$. Moreover,
$a_0 + V^+ = h_0$.
Indeed, the inclusion $a_0 + V^+ \supset h_0$ is clear, so one
only has to show that $u \in D(a_0 + V^+)$ implies $u \in D(h_0)$.
Without loss of generality assume that $u\ge0$. Let $(u_n)$ be a
sequence in $D(h_0)$ such that $u_n \to u$ in $D(a_0)$.
Then $D(h_0) \ni v_n := (\RRe u_n)^+\wedge u \to u$ in $D(a_0)$;
cf.\ the proof of Lemma~\ref{dense} and \cite[Proof of Lemma~3.13]{sv02}.
By the dominated convergence theorem
it follows that $v_n \to u$ in $D(h_{\max})$, so $u \in D(h_0)$.
\end{remark}

\section{Generation of quasi-contractive semigroups}\label{Scomplex2}

Throughout this section let $S_\Up$ be the $C_0$-semigroup on $L^p$
constructed in Proposition~\ref{appr}, for given $p\in J$ and $\mkU\in\U$,
and let $-\gen_\Up$ be the generator of~$S_\Up$.
At the end of the section we prove Theorem~\ref{mainresult2},
in which we eliminate the absorption potential $U$ via strong resolvent convergence.
In Theorem~\ref{illu} we give the proof of Theorem~\ref{mainresult2}
in a special case. 
We then need a modification of that proof to deduce the general case;
this involves resolvents twisted with suitable multiplication operators.

One of the key points that make the elimination of $U$ work is the following observation.

\begin{lemma}\label{Up-est}
Let $(X,\mu)$ be a measure space, let $p\in (1,\infty)$
and let $A$ be an $m$-accretive operator in $L^p(\mu)$.
Let $U\colon X\to [0,\infty)$ be measurable.
Suppose that $v_p(u) := u|u|^{\frac{p}{2}-1}\in \Q(U)$ and
\[
  U(v_p(u))\le \RRe \dual{Au}{w_p(u)}
\]
for all $u\in D(A)$, where $w_p(u) = u|u|^{p-2}$. Then
\[
  \|U^\frac1p (\lambda+A)^{-1}\|_{p\to p} \le \lambda^{-\frac1{p'}}
  \smallbds
\]
for all $\lambda>0$.
\end{lemma}
\begin{proof}
Let $f\in L^p(\mu)$ and set $u=(\lambda+A)^{-1}f$. Then
$\|u\|_p \le \lambda^{-1}\|f\|_p$ and
\begin{align*}
\|U^\frac1p u\|_p^p
&  = U(v_p(u))
 \le \RRe \dual{A u}{w_p(u)}
   = \RRe \dual{f}{w_p(u)}-\lambda \|u\|_p^p \\[0.5ex]
&\le \|f\|_p\|w_p(u)\|_{p'}
   = \|f\|_p\|u\|_p^{p-1}
 \le \lambda^{-({p-1})}\|f\|^p_p\,,
\end{align*}
which implies the assertion.
\end{proof}

Lemma~\ref{Up-est} will be used via the following result.

\begin{proposition}\label{strong-bound}
Let Assumptions~\ref{assump1}~--~\ref{assump8} be satisfied.
Let $\mkU\in\U$ and $p\in \sinter{J}$.
Let $U'\from\Omega\to[0,\infty)$ be measurable, let $c > 0$ and $\omega\in\R$, and
suppose that $U'(v) \le c\bigl(\tau_p(v)+\omega\|v\|_2^2\bigr)$ for all $v\in D(h_0)$.
Then
\begin{align*}
\bigl\|(U')^\frac1p(\lambda + \gen_\Up)^{-1}\bigr\|_{p\to p}
 &\le c^{\frac1p} (\lambda - \omega)^{-\frac1{p'}}, \\
\bigl\|(U')^\frac1{p'}(\lambda + \gen_\Up^*)^{-1}\bigr\|_{p'\to p'}
 &\le c^{\frac1{p'}} (\lambda - \omega)^{-\frac1p} \\[-1.3\baselineskip]
\end{align*}
for all $\lambda > \omega$.
\end{proposition}

\begin{proof}
It follows from the assumptions and the lower semicontinuity of the functional
$L^2 \ni v \mapsto \int U'|v|^2$
that $U'(v) \le c\bigl(\h_p(v)+\omega\|v\|_2^2\bigr)$ for all $v\in D(h_0)$.
Let $u\in D(\gen_\Up)$. By Proposition~\ref{appr} we obtain
$v := u|u|^{\frac p2 -1} \in D(h_0)$ and
\[
  \rfrac1c U'(v) \le \h_p(v)+\omega\|v\|_2^2
                 \le \RRe\dual{(\omega + \gen_\Up)u}{u|u|^{p-2}}.
  \avoidbreak
\]
Then the first assertion follows from Lemma~\ref{Up-est}.

To prove the second assertion, note that the semigroup $(e^{-t\gen_U^*})_{t\ge0}$
adjoint to $(e^{-t\gen_U})_{t\ge0}$ extrapolates to the
$C_0$-semigroup $(e^{-t\gen_\Up^*})_{t\ge0}$ on $L^{p'}$.
Moreover, $\gen_U^*$ is associated with the closure of $\form^*+U$,
where $\form^*$ denotes the adjoint form defined by
$\form^*(u,v) = \overline{\form(v,u)}$ on $D(\form^*) = D(\form)$.
Finally, $\form^*$ has the same structure as the form~$\form$:
\[
  \form^*(u,v)
  = \dual {\cA^* \nabla u}{\nabla v}
   + \dual{(-\overline{b_2}) \cdot \nabla u}{v}
   - \dual{(-\overline{b_1})\+u}{\nabla v}
   + \dual{\overline{Q}\+u}{v},
\]
from which one easily deduces that $\tau^*_\pstrich = \tau_p$,
where $\tau^*_\pstrich$ denotes the functional corresponding to $\form^*$ and $p'$;
cf.~\eqref{tau_p}.
Thus the second assertion follows from the same argument as the first one.
\end{proof}

Making use of Proposition~\ref{strong-bound},
we now give the proof of Theorem~\ref{mainresult2} in the special case where the
absorption potentials belong to $\U\cap L^1$.
(Observe that $\U \subset L^1$ if $\ind_\Omega\in D(h_0)$, which holds, e.g.,
if $\Omega$ is bounded and $D(h_0) = W^{1,2}$.)

\begin{theorem}\label{illu}
Let Assumptions \ref{assump1}~--~\ref{assump8} be satisfied, and let $p\in\sinter{J}$.
Let $U,U_1,U_2,\ldots\in\U \cap L^1$ be such that $U_m\le U$
for all $m \in \N$ and $U_m\to 0$ a.e.
Then the sequence $(S_{U_m\shift,p})_{m\in\N}$ converges in the semigroup sense.
\end{theorem}

\begin{proof}
By Proposition~\ref{appr} we know that $\|S_{U_m\shift,p}(t)\|_{p\to p}\le e^{\tomega_pt}$
for all $m\in\N$ and $t\ge0$.
We shall show that there exists an $\omega>0$ such that
for every $f\in L^1\cap L^\infty$,
the sequence $\bigl(\lambda(\lambda + \gen_{U_m\shift,p})^{-1}f\bigr)_{m\in\N}$
is convergent in $L^p$, uniformly for $\lambda \ge \nobreak 2\omega$.
Then it follows that the assumptions of the Trotter--Kato--Neveu
theorem~\cite[Theorem~IX.2.17]{kato80} are satisfied,
which yields the asserted semigroup convergence of $(S_{U_m\shift,p})_{m\in\N}$.

\pagebreak[1]

Without loss of generality we assume that $U>0$ a.e.
Let $q\in\sinter{J}$ with $q>p$. Since $\mkU\in\U$,
it follows from Lemma~\ref{lcomplex204} that there exist $c>0$ and $\omega\ge1$
such that $U(v) \le c\bigl(\tau_r(v)+\omega\|v\|_2^2\bigr)$ for all $v\in D(h_0)$
and $r\in\{p,q\}$.

In the next steps we fix $k,m\in\N$. Note that
\[
  U_{k,n} := U_k + (U_m-n)^+ \in \U, \quad
  U_{m,n} := U_m + (U_k-n)^+ \in \U
\]
for all $n\in\N$.
By Lemma~\ref{lp-conv} we see that
\[
  R_{j,n}(\lambda)
  := (\lambda + \gen_{U_{j,n}\shift,p})^{-1}
  \to (\lambda + \gen_{U_j\shift,p})^{-1}
  \avoidbreak
\]
strongly as $n\to\infty$, for all $\lambda>\omega$ and $j\in\{k,m\}$.

Let $\lambda > \omega$ and $n\in\N$.
Observe that $U_{k,n} - U_{m,n} = U_k\wedge n - U_m\wedge n$ is bounded.
Thus, $\gen_{U_{k,n}} = \gen_{U_{m,n}} + (U_{k,n} - U_{m,n})$ and hence
$\gen_{U_{k,n}\shift,p} = \gen_{U_{m,n}\shift,p} + (U_{k,n} - U_{m,n})$,
so the second resolvent equation gives
\[
  R_{m,n}(\lambda) - R_{k,n}(\lambda)
  = R_{k,n}(\lambda) (U_k\wedge n - U_m\wedge n) R_{m,n}(\lambda).
\]
Since $|U_k\wedge n - U_m\wedge n| \le |U_k-U_m|$,
it now follows from the second estimate in Proposition~\ref{strong-bound}
(applied to the adjoint operator from $L^p$ to $L^p$) that
\[
  \bigl\| R_{m,n}(\lambda)f - R_{k,n}(\lambda)f \bigr\|_p
  \le c^{\frac1{p'}} (\lambda - \omega)^{-\frac1p}
      \bigl\| U^{-\frac1{p'}} (U_k-U_m) R_{m,n}(\lambda)f\bigr\|_p\,.
\]
By H\"older's inequality and the first estimate in Proposition~\ref{strong-bound} we obtain
\begin{equation}\label{decomp}
\begin{split}
\bigl\| U^{-\frac1{p'}} (U_k-U_m) R_{m,n}(\lambda)f\bigr\|_p
 &\le \|U^{\frac1q} R_{m,n}(\lambda)f\|_q
      \|(U_k-U_m)U^{-\frac1{p'}-\frac1q}\|_{\frac{pq}{q-p}} \\
 &\le c^\frac1q (\lambda - \omega)^{-\frac1{q'}}  \|f\|_q
      \bigl\|\bigl(\tfrac{|U_k-U_m|}{U}\bigr)^{\frac{pq}{q-p}}U\bigr\|_1^{\frac1p-\frac1q},
\end{split}
\end{equation}
where we have used
$|U_k-U_m|U^{-\frac1{p'}-\frac1q} = \smash{\tfrac{|U_k-U_m|}{U}} U^\frac{q-p}{pq}$
in the second inequality.
With $C := c^{\frac1{p'}+\frac1q}$ we arrive at
\[
  \bigl\|R_{m,n}(\lambda)f - R_{k,n}(\lambda)f\bigr\|_p
  \le C(\lambda-\omega)^{-\frac1{q'}-\frac1p} \|f\|_q
      \bigl\|\bigl(\tfrac{|U_k-U_m|}{U}\bigr)^{\frac{pq}{q-p}}U\bigr\|_1^{\frac1p-\frac1q}.
\]

Now let $\lambda \ge 2\omega$, so that $\frac{\lambda}{\lambda-\omega} \le 2$.
Then we infer, letting $n\to\infty$, that
\[
  \lambda \bigl\|(\lambda + \gen_{U_m\shift,p})^{-1}f - (\lambda + \gen_{U_k\shift,p})^{-1}f\bigr\|_p
  \le 2C \omega^{\frac1q-\frac1p} \|f\|_q
  \bigl\|\bigl(\tfrac{U_k+U_m}{U}\bigr)^{\frac{pq}{q-p}}U\bigr\|_1^{\frac1p-\frac1q}.
\]
Since $\mkU\in L^1$, $\frac{U_m}{U_{}}\le 1$ and $\frac{U_m}{U}\to0$
a.e.\ as $m\to \infty$, we conclude that the sequence
$\bigl(\lambda(\lambda + \gen_{U_m\shift,p})^{-1}f\bigr)_{m\in\N}$
is convergent in $L^p$, uniformly for $\lambda \ge 2\omega$.
\end{proof}

\begin{remark}
In the case of uniform convergence $U_k/U\to0$ 
one can use~\eqref{decomp} with $p=q$ to obtain norm resolvent convergence.
\end{remark}

In general $\U\nsubseteq L^1$, and it is not even clear whether
$\U \cap L^1 \ne \varnothing$. However, since every $\mkU\in\U$ is form
bounded with respect to $h_0$, one has $\mkU\rho^2\in L^1$ for all
$\rho\in D(h_0)$. This is the basic observation for adapting
the technique of Theorem~\ref{illu} to the general case.
Let $\rho \in D(h_0)$ satisfy $\rho>0$ a.e., and set $\eps = 2\rfrac{q-p}{pq}$.
Then instead of using~\eqref{decomp} we will work with the inequality
\begin{equation}\label{decompeps}
\begin{split}
\MoveEqLeft[4]
\bigl\| U^{-\frac1{p'}} (U_k-U_m) R_{m,n}(\lambda)f \bigr\|_p \\
 &\le \bigl\| U^{\frac1q}\rho^{-\eps}R_{m,n}(\lambda)f \bigr\|_q
      \bigl\| (U_k-U_m) U^{-\frac1{p'}-\frac1q} \rho^\eps \bigr\|_{\frac{pq}{q-p}} \\
 &\le \bigl\| U^{\frac1q}\rho^{-\eps}R_{m,n}(\lambda)\rho^\eps \bigr\|_{q\to q}
      \|\rho^{-\eps}f\|_q
      \bigl\| \bigl(\tfrac{|U_k-U_m|}{U}\bigr)^{\frac{pq}{q-p}} \+ U\rho^2 \bigr\|_1^{\frac1p-\frac1q},
\end{split}
\end{equation}
where $\rho^\eps$ and $\rho^{-\eps}$ are understood as multiplication operators,
and we assume that $\rho^{-\eps} f \in L^q$.
This idea motivates us to study the twisted resolvent
$\rho^{-\eps}(\lambda + \gen_U)^{-1}\rho^\eps$ and the corresponding sesquilinear form,
for given $\mkU\in\U$.

As a first preparation we investigate under which conditions $\rho^\eps$ 
is a bounded multiplication operator on the form domain.

\begin{lemma}\label{multop}
Let $(X,\mu)$ be a measure space, let $\h$ be a symmetric Dirichlet form in $L^2(\mu)$,
and let $H$ be the associated positive self-adjoint operator in $L^2(\mu)$.
\begin{alplist}[i]
\item \label{multop(a)} 
Let $f \in L^2 \cap L^\infty$ and $\rho := (I+H)^{-1} f$.
Then $\rho$ is a bounded multiplication operator on $D(\h)$.

\item \label{multop(b)} 
Let $\rho\in L^\infty(\mu)$ be a bounded multiplication operator on $D(\h)$,
and let $F\from \C\to\C$ be Lipschitz continuous. Then $F\circ\rho$ is a
bounded multiplication operator on $D(\h)$.
\end{alplist}
\end{lemma}
\begin{proof}
\ref{multop(a)}
Assume without loss of generality that $f$ and hence $\rho$ is real-valued.
We shall show that there exists a constant $c\ge0$ such that
$(\h+1)(\rho u) \le c(\h+1)(u)$ for all real-valued $u \in D(\h)\cap L^\infty$;
then the assertion follows since $D(\h)\cap L^\infty$ is a core for~$\h$
and the form $\h$ is real.

Observe that
$\rho u,\rho u^2 \in D(\h)$ since $\rho,u \in D(\h)\cap L^\infty$.
For all $x,y\in X$ we have $(\rho u)(x)^2 = \rho(x)\cdot(\rho u^2)(x)$ and
\begin{align*}
\MoveEqLeft
\bigl((\rho u)(x)-(\rho u)(y)\bigr)^2 \\
  &= \rho(x)\rho(y) \bigl(u(x)-u(y)\bigr)^2
     + \bigl(\rho(x)-\rho(y)\bigr) \bigl((\rho u^2)(x)-(\rho u^2)(y)\bigr).
\end{align*}
An application of \cite[Proposition~2.1]{hua02} yields
\[
  \dual{(I-e^{-tH})(\rho u)}{\rho u}
  \le \|\rho\|_\infty^2 \dual{(I-e^{-tH})u}{u} + \dual{(I-e^{-tH})\rho}{\rho u^2}
\]
for all $t>0$, and hence
\[
  (\h+1)(\rho u) \le \|\rho\|_\infty^2 \h(u) + (\h+1)(\rho,\rho u^2)
  = \|\rho\|_\infty^2 \h(u) + \dual{f}{\rho u^2}.
  \avoidbreak
\]
This completes the proof of \ref{multop(a)} since $f\rho\in L^\infty(\mu)$.

\ref{multop(b)}
It is not too difficult to see that the function
$\Phi \from \set{(w,z)\in\C\times\C}{|z|\le\|\rho\|_\infty |w|} \to \C$
defined by $\Phi(w,z) = F(\rfrac zw)w$ if $(w,z)\ne(0,0)$ and
$\Phi(0,0)=0$ is Lipschitz continuous. Moreover,
$(F\scirc\rho)\mkern1mu u = \Phi\circ(u,\rho u)$ for all $u\in D(\h)$,
so the assertion follows from \cite[Theorem~I.4.12]{mr92}
applied to $\RRe\Phi$ and to $\IIm\Phi$.
\end{proof}

\medskip

In the following let $\rho \in W^{1,1}_\loc \cap L^\infty$ be
a bounded multiplication operator on $D(h_0)$ with $\rho>0$ a.e.\ 
We further suppose that there exist $\kappa, K \ge 0$ such that
\begin{equation}\label{Wrho}
W_\lrho := \frac{\< \cA_0^\s \nabla \rho,\nabla \rho\>}{\rho^2} \le \kappa h_0 + K.
\end{equation}
We first consider the special case where $\rho^{-1}\in L^\infty$;
then $\rho^{-1} \in W^{1,1}_\loc$.
Let $\eps\in\R$.
Since $\rho^{\pm1}\in L^\infty$, there exists a Lipschitz continuous function
$F\from \C\to\C$ such that $F\circ\rho =\rho^\eps$.
Hence by Lemma~\ref{multop}\ref{multop(b)},
$\rho^\eps$ is a bounded multiplication operator on $D(h_0)$.
It follows that
$(u,v) \mapsto \form(\rho^\eps u, \rho^{-\eps} v)$ is a bounded form on $D(h_0)$.
A~straightforward computation yields
\begin{equation} \label{t-eps}
\begin{split}
\form(\rho^\eps u, \rho^{-\eps} v) 
= \form(u,v)
 & - \eps\dual{\cA\nabla u}{\rfrac{\nabla\rho}{\rho}\+v}
   + \eps\dual{\cA \rfrac{\nabla\rho}{\rho}\+u}{\nabla v} \\[0.8ex]
 & + \eps \dual{\inner{\bplusb}{\rfrac{\nabla\rho}{\rho}}\+u}{v}
   - \eps^2\dual{\tfrac{\< \cA \nabla\rho,\nabla\rho\>}{\rho^2}\+u}{v}
\end{split}
\end{equation}
for all $u,v \in D(h_0)$, so the form is of the same structure as $\form$,
with new lower-order coefficients
\[
  b_1^{(\eps)} = b_1 - \eps\cA^\top\rfrac{\nabla\rho}{\rho}, \quad
  b_2^{(\eps)} = b_2 - \eps\cA\rfrac{\nabla\rho}{\rho}, \quad
  Q^{(\eps)} = Q + \eps\inner{\bplusb}{\rfrac{\nabla\rho}{\rho}}
               - \eps^2\tfrac{\< \cA \nabla\rho,\nabla\rho\>}{\rho^2}\,.
\]

We define a new reference form $h_0^{(\eps)}$ with domain $D(h_0^{(\eps)}) = D(h_0)$ by
\[
  h_0^{(\eps)}(u,v)
  = \dual{\cA_0^\s \nabla u}{ \nabla v} + V_\eps^+(u,v),
\]
where $V_\eps = \RRe Q^{(\eps)}$. Whenever the tuple
\smash{$\bigl(\cA, b_1^{(\eps)}, b_2^{(\eps)}, Q_{\phantom0}^{(\eps)}, h_0^{(\eps)}\bigr)$}
satisfies Assumptions \ref{assump1}~--~\ref{assump4},
we denote by $\form^{(\eps)}$ the sesquilinear form
and by $\tau_p^{(\eps)}$ ($p \in [1,\infty]$) the functionals
associated with the new lower-order coefficients.

In the following let $\alpha_\s$ be as in~\eqref{alpha-s}.
By Assumption~\ref{assump3} there exists an $M\ge0$ such that
$\<\cA_0^\a\xi, \eta\>^2 \le M^2 \<\cA_0^\s\xi, \xi\> \<\cA_0^\s\eta, \eta\>$
for all $\xi, \eta \in \R^N$.
Moreover, $\mkUhat \le \hat c(h_0 + 1)$ for some $\hat c>0$ by Assumption~\ref{assump4}.
In the next result we will use the constants
\begin{equation} \label{Cp-def}
  C_p := 2 \bigl( M^2 + (1-\rfrac2p)^2 + \hat c \bigr) + \alpha_\s^2
  \qquad (p\in[1,\infty]).
\end{equation}

\begin{lemma}\label{form-weighted}
Let Assumptions \ref{assump1}~--~\ref{assump4} be satisfied, and let
$\rho\in W^{1,1}_\loc \cap L^\infty$ be a bounded multiplication operator on $D(h_0)$.
Assume that $\rho>0$ a.e., $\rho^{-1}\in L^\infty$ and that $\rho$ satisfies~\eqref{Wrho} for some $\kappa, K\ge 0$.
Let $p\in[1,\infty]$.
\begin{alplist}[i]
\item \label{form-weighted(a)}
Let $C := 1 + (C_1+1)\max\{\kappa,K\}$,
and let $0 < \eps \le \eps_0 := \min\{1,\frac{1}{2C}\}$.
Then $\tfrac{1}{2} h_0 - 1 \le h_0^{(\eps)} \le 2h_0+1$,
the tuple $\bigl(\cA, b_1^{(\eps)}, b_2^{(\eps)}, Q_{\phantom0}^{(\eps)}, h_0^{(\eps)}\bigr)$
satisfies Assumptions \ref{assump1}~--~\ref{assump4}, and
\[
  \tau_p^{(\eps)}(v)
  \ge \tau_p(v) - \eps C(h_0 +1)(v)
  \smallbds
\]
for all $v\in D(h_0)$.
\item \label{form-weighted(b)} 
Suppose that $\kappa = 0$, and let $\eps > 0$.
Then the tuple $\bigl(\cA, b_1^{(\eps)}, b_2^{(\eps)}, Q_{\phantom0}^{(\eps)}, h_0^{(\eps)}\bigr)$
satisfies Assumptions \ref{assump1}~--~\ref{assump4}, and
\[
  \tau_p^{(\eps)}(v)
  \ge \tau_p(v) - \delta (h_0+1)(v)
    - \eps^2 K \bigl(\tfrac{C_p}\delta + 1\bigr) \|v\|_2^2
\]
for all $\delta>0$ and $v\in D(h_0)$.
\end{alplist}
\end{lemma}

\begin{proof}
First observe that the form \smash{$h_0^{(\eps)}$} satisfies Assumption~\ref{assump2}
if there exists a $c_\eps\ge0$ such that
$|h_0^{(\eps)}(v) - h_0(v)| \le \frac12 h_0(v) + c_\eps\|v\|_2^2$ for all $v\in D(h_0)$.
Then, using this estimate and~\eqref{Wrho}, one easily shows that the coefficients
$b_1^{(\eps)}, b_2^{(\eps)}, Q_{\phantom0}^{(\eps)}$ satisfy Assumption~\ref{assump4}.

We shall show that
\begin{align}
\label{h0-eps_est}
|h_0^{(\eps)}(v) - h_0(v)|
 &\le \delta(h_0+1)(v) + \eps^2 \bigl( \tfrac{\hat c}{\delta}+1 \bigr) W_\lrho(v), \\[0.5ex]
\label{taueps_est}
\tau_p^{(\eps)}(v)
 &\ge \tau_p(v) - \delta(h_0+1)(v) - \eps^2 \bigl( \tfrac{C_p}{\delta}+1 \bigr) W_\lrho(v)
\end{align}
for all $v\in D(h_0)$ and $\delta>0$.
Then one easily obtains the assertions of~\ref{form-weighted(a)} by choosing $\delta = \eps$
and noting that $\hat c \le C_p \le C_1$, $W_\lrho \le \max\{\kappa,K\}(h_0+1)$ and
\[
  \eps + \eps^2 \bigl( \rfrac{C_p}{\eps}+1 \bigr) \max\{\kappa,K\}
  \le \eps C \le \tfrac12\,,
\]
provided $0<\eps\le\eps_0$.
If $\kappa=0$, then we choose $\delta = \frac12$ in~\eqref{h0-eps_est},
and the assertions of~\ref{form-weighted(b)} follow since $W_\lrho \le K$.

It remains to verify~\eqref{h0-eps_est} and~\eqref{taueps_est}.
Note that
\begin{equation}\label{h0-eps-diff}
  |h_0^{(\eps)}(v) - h_0(v)| = |V_\eps^+(v) - V^+(v)| \le |V_\eps(v) - V(v)|
\end{equation}
for all $v \in D(h_0)$.
Since $V_\eps = V\mkern-2mu + \eps\inner{\RRe(\bplusb)}{\rfrac{\nabla\rho}{\rho}} 
- \eps^2W_\lrho$, we can estimate
\begin{equation}\label{V-eps-diff}
|V_\eps-V|
 \le \eps \bigl(4\Uhat\bigr)^{1/2} W_\lrho^{1/2} + \eps^2 W_\lrho
 \le \tfrac{\delta}{\hat c}\mkUhat + \eps^2 \tfrac{\hat c}{\delta} W_\lrho + \eps^2 W_\lrho
\end{equation}
by the Cauchy--Schwarz inequality and the definition~\eqref{Uhat} of $\Uhat$.
By~\eqref{h0-eps-diff} and the estimate $\frac{1}{\hat c}\mkUhat \le h_0+1$,
this implies~\eqref{h0-eps_est}.

Observe that $\RRe\bigl( -\frac1p\cA^\top\! + \frac{1}{p'}\cA \bigr)
= (1-\rfrac2p)\cA_0^\s + \cA_0^\a$ and
$\IIm\bigl( \cA^\top\! + \cA \bigr) = 2\cA_1^\s$,
where the real and imaginary parts are taken coefficient-wise.
It follows that
\[
  \RRe\bigl( \tfrac1p b_1^{(\eps)} \mkern-2mu- \tfrac1{p'} b_2^{(\eps)} \bigr)
  = \RRe\bigl( \tfrac1p b_1 - \tfrac1{p'} b_2 \bigr)
    + \eps \bigl( (1-\rfrac2p)\cA_0^\s + \cA_0^\a \bigr) \rfrac{\nabla\rho}{\rho}
  \smallbds
  \postdisplaypenalty10
\]
and
\[
  \IIm(b_1^{(\eps)} \mkern-2mu+ b_2^{(\eps)})
  = \IIm(\bplusb) - 2\eps \cA_1^\s \rfrac{\nabla\rho}{\rho}\,.
\]
Thus we obtain
\begin{equation}\label{taueps}
\smallbds
\begin{split}
\tau_p^{(\eps)}(v)
= {} & \tau_p(v) + 2\eps \dual{|v|\bigl((1-\rfrac2p)\cA_0^\s+\cA_0^\a\bigr) \rfrac{\nabla\rho}{\rho}}{\nabla |v|}
       \\[0.5ex]
 & + 2\eps\dual{|v|\cA_1^\s \rfrac{\nabla\rho}{\rho}}{\IIm(\overline{\sgn v}\,\nabla v)}
   + (V_\eps-V)(v)
\end{split}
\end{equation}
for all $v\in D(h_0)$.
\pagebreak[1]

Let $v\in D(h_0)$, and set $\eta = \IIm(\overline{\sgn v}\,\nabla v)$.
For all $\xi \in \R^N$, the definition of $M$ implies that \vspace{-0.5ex}
\[
  \<\cA_0^\a\xi, (\cA_0^\s)^{-1}\cA_0^\a\xi\>^2
  \le M^2 \<\cA_0^\s\xi, \xi\> \<\cA_0^\a\xi, (\cA_0^\s)^{-1}\cA_0^\a\xi\>
  \smallbds
\]
and hence
\begin{align*}
\inner{(\cA_0^\s)^{-1}(\alpha\cA_0^\s+\cA_0^\a)\xi}
                     {(\alpha\cA_0^\s+\cA_0^\a)\xi}
 &= \alpha^2 \<\cA_0^\s\xi, \xi\> + \<(\cA_0^\s)^{-1}\cA_0^\a\xi, \cA_0^\a\xi\>
    \\[0.5ex]
 &\le (\alpha^2+M^2) \<\cA_0^\s\xi, \xi\>
\end{align*}
for all $\alpha\in\R$.
By the Cauchy--Schwarz inequality we infer that
\begin{align*}
\MoveEqLeft[3]
\left| \inner{|v|\bigl((1-\rfrac2p)\cA_0^\s+\cA_0^\a\bigr)
              \rfrac{\nabla\rho}{\rho}}{\nabla |v|} \right|^2 \\[0.5ex]
&\le \bigl( (1-\rfrac2p)^2 + M^2 \bigr)
     \inner{\cA_0^\s\rfrac{\nabla\rho}{\rho}}{\rfrac{\nabla\rho}{\rho}} |v|^2
     \cdot \inner{\cA_0^\s\nabla |v|}{\nabla |v|}. \\[-1.3\baselineskip]
\end{align*}
Similarly,
\[
  \bigl| \inner{|v|\cA_1^\s\rfrac{\nabla\rho}{\rho}}{\eta} \bigr|^2
  \le \alpha_s^2
      \inner{\cA_0^\s\rfrac{\nabla\rho}{\rho}}{\rfrac{\nabla\rho}{\rho}} |v|^2
      \cdot \inner{\cA_0^\s \eta}{\eta}
\]
by~\eqref{alpha-s}.
Together with~\eqref{V-eps-diff} it follows from~\eqref{taueps} that
\begin{align*}
\tau_p^{(\eps)}(v)
 \ge {}& \tau_p(v)
       - 2\eps \bigl( M^2 + (1-\rfrac2p)^2 \bigr)^\frac12
         \amax(|v|)^\frac12 W_\lrho(v)^\frac12 \\
     & - 2\eps \alpha_\s \dual{\cA_0^s\eta}{\eta}^\frac12 W_\lrho(v)^\frac12
       - \bigl( \tfrac{\delta}{2\hat c}\Uhat
                + \eps^2 \tfrac{2\hat c}{\delta} W_\lrho + \eps^2 W_\lrho \bigr)(v).
\end{align*}
Finally, observe that $\amax(v) \le h_0(v)$ and
$\dual{\cA_0^s\eta}{\eta} = \amax(v) - \amax(|v|) = h_0(v) - h_0(|v|)$.
Then, estimating
\begin{align*}
2\eps \bigl( M^2 + (1-\rfrac2p)^2 \bigr)^\frac12 \amax(|v|)^\frac12 W_\lrho(v)^\frac12
 &\le \lfrac\delta2 h_0(|v|)
      + \eps^2 \lfrac2\delta \bigl( M^2 + (1-\rfrac2p)^2 \bigr) W_\lrho(v) \\
2\eps \alpha_\s \dual{\cA_0^s\eta}{\eta}^\frac12 W_\lrho(v)^\frac12
 &\le \delta \bigl( h_0(v) - h_0(|v|) \bigr)
      + \eps^2 \lfrac1\delta \alpha_\s^2 W_\lrho(v)
\end{align*}
and $\frac{\delta}{2\hat c}\mkUhat(v) \le \frac\delta2 (h_0+1)(|v|)$,
we conclude that~\eqref{taueps_est} holds, and the proof is complete.
\end{proof}

We point out that in the following result we do not assume $\rho^{-1}\in L^\infty$.

\begin{proposition}\label{weight-Um}
Let Assumptions~\ref{assump1}~--~\ref{assump8} be satisfied,
and let $\rho\in W^{1,1}_\loc \cap L^\infty$ be
a bounded multiplication operator on $D(h_0)$.
Assume that $\rho>0$ a.e.\ and that $\rho$ satisfies~\eqref{Wrho}
for some $\kappa, K\ge 0$.
Let $p\in \inter{J}$, and let $\mu_p>0$ and $\omega_p\in\R$ be as in~\eqref{mu-omega}.
\begin{alplist}[i]
\item \label{weight-Um(a)} 
Let $U'\from\Omega\to[0,\infty)$ be measurable,
and suppose that $U'\le c_1(h_0+1)$ for some $c_1>0$.
Then there exist $\delta,\overline{\eps} > 0$ such that
for all $\eps\in (0,\overline{\eps}]$, $\lambda>\omega_p+\mu_p$, $\mkU\in\U$
and $q\in J$ with \smash[b]{$|\rfrac1p-\rfrac1q|\le\delta$} one has
\begin{equation}\label{U'rho-eps}
  \|(U')^\frac1q \rho^{-\eps}(\lambda+\gen_\Up)^{-1}\rho^\eps\|_{q\to q}
  \le \bigl(\rfrac{2c_1}{\mu_p}\bigr)^{\frac1q}
      (\lambda-\omega_p-\mu_p)^{-\frac1{q'}}.
\end{equation}

\item \label{weight-Um(b)}
Suppose that $\kappa=0$. 
Then
\[
  \|\rho^{-\eps}S_\Up(t)\rho^\eps\|_{p\to p}
  \le \exp\Bigl(\mkern-1.5mu \bigl( \omega_p+\mu_p+\eps^2K(\rfrac{C_p}{\mu_p}+1) \bigr)t \Bigr)
\]
for all $\eps>0$, $\mkU\in\U$ and $t\ge0$, with $C_p$ from~\eqref{Cp-def}.
\end{alplist}
\end{proposition}

\begin{proof}
\ref{weight-Um(a)}
We first assume that $\rho^{-1}\in L^\infty$ and that $U \ge \mkUhat-c$ for some $c\in\R$.
In the last step of the proof we will remove these assumptions.
\pagebreak[1]

Let $\eps_0 > 0$ and $C \ge 1$ be as in Lemma~\ref{form-weighted}\ref{form-weighted(a)},
and fix $\eps \in (0,\eps_0]$.
Recall from~\eqref{t-eps} that the form $\form^{(\eps)}$ with lower-order coefficients
$b_1^{(\eps)}, b_2^{(\eps)}, Q^{(\eps)}$ is a bounded form on $D(h_0)$ given by
\begin{equation} \label{form-rho}
  \form^{(\eps)}(u,v) = \form(\rho^\eps u, \rho^{-\eps} v).
\end{equation}
We will employ the corresponding Dirichlet form $h_0^{(\eps)}$ and
the functionals $\tau_q^{(\eps)}$ ($q \in [1,\infty]$) as in Lemma~\ref{form-weighted};
in addition we use the corresponding set $\U^{(\eps)}$ of potentials
and the interval $J^{(\eps)}$.

By~\eqref{tau_p} we have $\RRe \form(v) = \tau_2(v)$ for all $v \in D(h_0)$,
and similarly for $\form^{(\eps)}$.
Thus, by Lemma~\ref{form-weighted}\ref{form-weighted(a)} we obtain
$\RRe\form^{(\eps)} \ge \RRe\form - \eps C (h_0+1)$.
It follows from \eqref{t+Uhat} and~\eqref{mu-omega} that $\RRe(\form+\Uhat) \ge \mu_p h_0 - \omega_p$.
Together with $U \ge \mkUhat-c$ we infer that
\begin{equation} \label{form-eps+U}
  \RRe (\form^{(\eps)} + U)
  \ge (\mu_p - \eps C) h_0 - (\omega_p+c+\eps C).
\end{equation}
Therefore if $\eps < \frac{\mu_p}{C}$, then $\form^{(\eps)} + U$ is 
a closed sectorial form; in particular, $\mkU \in \U^{(\eps)}$.
It follows from~\eqref{form-rho} that the form $\form^{(\eps)} + U$ is associated
with the operator $\rho^{-\eps} \gen_U \rho^\eps$ with domain $\rho^{-\eps} D(\gen_U)$
and that the corresponding $C_0$-semigroup is
$(\rho^{-\eps}e^{-t\gen_U}\mkern-2mu\rho^\eps)_{t\ge0}$.

By \eqref{tau-lip} there exists a $C_0>0$ such that
$|\tau_p(v)-\tau_q(v)|\le C_0 |\rfrac1p-\rfrac1q|(h_0+1)(v)$ for all
$v\in D(h_0)$ and $q\in [1,\infty]$.
Set $\delta = \frac{\mu_p}{4C_0}$ and fix
$q\in [1,\infty]$ with $|\rfrac1p-\rfrac1q| \le \delta$.
If $\eps \le \overline \eps := \min\bigl\{\eps_0,\frac{\mu_p}{4C}\bigr\}$,
then by Lemma~\ref{form-weighted}\ref{form-weighted(a)} and~\eqref{mu-omega} we obtain
\begin{align*}
\tau_q^{(\eps)}(v) 
 &\ge \tau_q(v) - \eps C (h_0 + 1)(v)
  \ge \tau_p(v) - (C_0 \delta + \eps C) (h_0 + 1)(v) \\[0.5ex minus 0.3ex]
 &\ge \tfrac{1}{2} \mu_p h_0(v) - (\omega_p + \tfrac{1}{2} \mu_p) \|v\|_2^2
\end{align*}
for all $v \in \smash{D(h_0)}$.
On the one hand, this implies that $q$ is in the interior of $J^{(\eps)}$,
due to Lemma~\ref{lcomplex204} and the estimate $h_0 \ge \frac12(h_0^{(\eps)}-1)$.
On the other hand we infer, using $U' \le c_1(h_0+1)$, that
\[
  U'(v) \le c_1\rfrac{2}{\mu_p} \bigl( \tau_q^{(\eps)}(v)
            + (\omega_p + 2\cdot\tfrac12 \mu_p) \|v\|_2^2 \mkern1mu \bigr)
  \avoidbreak
\]
for all $v \in D(h_0)$.
Thus we have verified the conditions of Proposition~\ref{strong-bound}
for the perturbed operator $\rho^{-\eps} \gen_U \rho^\eps$
associated with the form $\form^{(\eps)} + U$, and~\eqref{U'rho-eps} follows
in the case that the initial assumptions on $\rho$ and $U$ hold.

Now we prove the assertion for general $\rho$ and $\mkU$.
We define $\rho_k := \rho\vee \frac1k$ and $U_m := \mkU+(\mkUhat-m)^+$;
then $\rho_k^{-1} \in L^\infty$, $W_{\lrho_k} \le W_\lrho \le \kappa h_0 + K$ for all $k\in\N$
and $U_m \in \U$, $U_m \ge \mkUhat-m$ for all $m\in\N$,
so by the above it follows that
\begin{equation} \label{km}
  \bigl\|(U')^\frac1q\rho_k^{-\eps}(\lambda + \gen_{U_m\shift,p})^{-1}\rho_k^\eps\bigr\|_{q\to q}
  \le \bigl(\rfrac{2c_1}{\mu_p}\bigr)^{\frac1q} (\lambda-\omega_p-\mu_p)^{-\frac1{q'}}
\end{equation}
for all $\lambda > \omega_p+\mu_p$ and $q \in [1,\infty]$ with 
$|\rfrac1p-\rfrac1q| \le \delta$.
By Lemma~\ref{lp-conv} we see that $S_{U_m\shift,p} \to S_\Up$
in the semigroup sense on $L^p$ as $m\to\infty$.
Moreover, up to a subsequence, 
$(\lambda+\gen_\Up)^{-1}\rho_k^\eps f \to (\lambda+\gen_\Up)^{-1}\rho^\eps f$
a.e.\ as $k \to \infty$, for all $f\in L^p\cap L^q$.
Thus we obtain the assertion by first letting $m\to\infty$ in~\eqref{km},
and then $k\to\infty$, taking into account Fatou's lemma.

\ref{weight-Um(b)}
The proof is similar as above; we only point out the main differences.
Fix $\eps>0$. Instead of~\eqref{form-eps+U} we obtain
\[
  \RRe (\form^{(\eps)} + U)
  \ge \tfrac{\mu_p}{2} h_0
      - \bigl( \omega_p + c + \tfrac{\mu_p}{2} + \eps^2K(\rfrac{2C_p}{\mu_p}+1) \bigr),
\]
by Lemma~\ref{form-weighted}\ref{form-weighted(b)}
applied with $\delta = \frac{\mu_p}{2}$.
As above it follows that $\mkU \in \U^{(\eps)}$.

Applying Lemma~\ref{form-weighted}\ref{form-weighted(b)} with $\delta=\mu_p$
we obtain
\[
  \tau_p^{(\eps)}(v)
  \ge  - \bigl(\omega_p + \mu_p
    + \eps^2K(\rfrac{C_p}{\mu_p}+1) \bigr)\|v\|_2^2
\]
for all $v \in D(h_0)$.
Hence $p \in J^{(\eps)}$, and the assertion follows from Proposition~\ref{appr}.
\end{proof}

The final step in the preparation of the proof of Theorem~\ref{mainresult2}
is the existence of a weight $\rho$
that satisfies the assumptions of Proposition~\ref{weight-Um}
and belongs to $D(h_0)$, not just $W^{1,1}_\loc$.

\begin{lemma}\label{lcompl230}
Let $f \in L^2$ be such that $0 < f \le 1$ a.e.
Set $\rho = (I + H_0)^{-1} f$, where $H_0$ is the
positive self-adjoint operator associated with the Dirichlet form~$h_0$.
Then $0 < \rho \le 1$ a.e., $\rho \in D(h_0)$,
$\rho$ is a bounded multiplication operator on $D(h_0)$ and
\[
  W_\lrho = \frac{ \< \cA_0^\s \nabla \rho, \nabla \rho \>}{\rho^2} \le 4 h_0 + 2
  \smallbds
\]
as forms.
\end{lemma}
\begin{proof}
By Lemma~\ref{proposition-app} below we have $\rho > 0$ a.e.,
and $\rho\le1$ a.e.\ since $h_0$ is a Dirichlet form.
Moreover, $\rho \in D(H_0) \subset D(h_0)$.
It follows from Lemma~\ref{multop}\ref{multop(a)} that $\rho$ is a
bounded multiplication operator on $D(h_0)$.
We shall show that $h_0(u) \ge \frac14 W_\lrho(u) - \frac12 \|u\|_2^2$
for all real-valued functions $u \in D(h_0) \cap L^\infty$;
then the last assertion follows since $D(h_0) \cap L^\infty$
is a core for $h_0$ and the form $h_0$ is real.

Let $\eps > 0$. Then the function $x \mapsto (x^++\eps)^{-1/2}$
is bounded and Lipschitz continuous on $\R$. By Lemma~\ref{multop}\ref{multop(b)} it follows
that $v := (\rho+\eps)^{-1/2} u \in D(h_0) \cap L^\infty$.
Moreover, $\nabla u =
(\rho+\eps)^{1/2} \+ \nabla v + \frac12 v (\rho+\eps)^{-1/2} \+ \nabla \rho$.
Therefore,
\begin{align*}
\< \cA_0^\s \nabla u, \nabla u \>
 &= (\rho + \eps) \< \cA_0^\s \nabla v, \nabla v \>
    + \< \cA_0^\s \nabla \rho, v\nabla v \>
    + \tfrac14 \+ (\rho + \eps)^{-1} \+ v^2 \+ \< \cA_0^\s \nabla \rho, \nabla \rho \> \\[0.8ex]
 &\ge \tfrac12 \+ \< \cA_0^\s \nabla \rho, \nabla v^2 \>
    + \tfrac14 \+ (\rho + \eps)^{-2} \+ u^2 \+ \< \cA_0^\s \nabla \rho, \nabla \rho \> ,
\end{align*}
because $\< \cA_0^\s \nabla v, \nabla v \> \ge 0$.
Since $v^2 \in D(h_0)$ and $V^+(u) \ge \frac12 V^+(\rho,v^2)$, we obtain
\begin{equation}\label{h0u}
  h_0(u) = \dual{\cA_0^\s \nabla u}{\nabla u} + V^+(u)
        \ge \frac14 \int\frac{\<\cA_0^\s \nabla \rho, \nabla \rho\>}{(\rho+\eps)^2}u^2
        + \frac12 h_0(\rho,v^2) .
\end{equation}
Moreover,
\[
  h_0(\rho,v^2) = \dual{H_0\rho}{v^2}
  = \dual{f-\rho}{v^2} \ge -\dual{\rho}{v^2} \ge - \|u\|_2^2\,.
\]
Letting $\eps\to0$ in~\eqref{h0u}, we conclude that
$h_0(u) \ge \frac14 W_\lrho(u) - \frac12 \|u\|_2^2$,
and the proof is complete.
\end{proof}

\begin{lemma}\label{proposition-app}
Let $p\in[1,\infty)$, and
let $B\in \mathcal{L}(L^p)$ be a positive operator with dense range.
Then $Bf>0$ a.e.\ for all $f\in L^p$ such that $f>0$ a.e.
\end{lemma}

\begin{proof}
Let $E$ be a measurable set such that $Bf=0$ a.e.\ on $E$.
Let $g \in L^p$ satisfy $g\ge 0$.
Then $0 \le B(g\wedge (nf)) \le n Bf$,
so $B(g\wedge (nf)) = 0$ a.e.\ on $E$ for all $n \in \N$.
Moreover, $B(g\wedge (nf)) \to Bg$ in $L^p$.
Therefore $Bg = 0$ a.e.\ on $E$.
By linearity, $Bg=0$ a.e.\ on $E$ for all $g\in L^p$.
Since $R(B)$ is dense in $L^p$, we conclude that $E$ is a null set.
\end{proof}

Now we are ready to prove our main result on the existence of a
quasi-contractive $C_0$-semigroup associated with $\cL$.

\begin{proof}[Proof of Theorem~\ref{mainresult2}]
It was proved in Lemmas~\ref{lcomplex204} and \ref{lcomplex202}\ref{lcomplex202-1} that 
$\sinter{J}\ne\varnothing$ and $\U \neq \varnothing$.
Now let $U,U_1,U_2,\ldots \in\U$ satisfy
$U_m\le U$ for all $m\in\N$ and $U_m\to0$ a.e.\ as $m\to\infty$.
Without loss of generality we assume that $U>0$ a.e.
By Proposition~\ref{appr}, the estimate
\begin{equation}\label{contr-p}
\|S_{U_m\shift,r}(t)\|_{r\to r} \le e^{\tomega_r t}
\avoidbreak
\end{equation}
holds for all $m\in\N$, $r\in J$ and $t \ge 0$.

Let $p \in \inter{J}$.
It follows from Lemma~\ref{lcompl230} that there exists a 
$\rho \in D(h_0) \cap L^\infty$ such that $\rho > 0$ a.e., $\rho$ is 
a bounded multiplication operator on $D(h_0)$ and $W_\lrho \le 4 h_0 + 2$.
Hence the assumptions of
Proposition~\ref{weight-Um}\ref{weight-Um(a)} are satisfied with $\kappa=4$ and $K=2$.
As a result there exist
$\delta,\overline{\eps},C_0,\omega_0 > 0$ such that
\begin{equation}\label{twist}
\|U^\frac1q \rho^{-\eps}(\lambda+\gen_{\sUtilde\mkern-3mu,\mkern1mup})^{-1}\rho^\eps\|_{q\to q}
  \le C_0(\lambda-\omega_0)^{-\frac1{q'}}
\avoidbreak
\end{equation}
for all $\eps \in (0,\overline{\eps}]$, $\lambda>\omega_0$, $\sUtilde\in\U$
and all $q\in J$ with $|\rfrac1q-\rfrac1p|\le\delta$.

Choose $q\in J$ such that $0 < \rfrac1p - \rfrac1q \le \frac{\overline{\eps}}2\wedge\delta$.
Then $\eps := 2\bigl(\rfrac1p - \rfrac1q\bigr) \le \overline{\eps}$.
Observe that $M_\eps := \set{f \in L^1\cap L^\infty}{\rho^{-\eps}f \in L^1\cap L^\infty}$
is dense in $L^p$.

Now we prove that the sequence $(S_{U_m\shift,p})_{m \in \N}$ converges
in the semigroup sense.
Let $f\in M_\eps$. Following the argument of Theorem~\ref{illu},
replacing \eqref{decomp} with~\eqref{decompeps} and applying~\eqref{twist},
one finds $C,\omega>0$ such that
\[
  \lambda \bigl\|(\lambda + \gen_{U_m\shift,p})^{-1}f - (\lambda + \gen_{U_k\shift,p})^{-1}f\bigr\|_p
  \le C \|\rho^{-\eps}f\|_q
  \bigl\| \bigl(\tfrac{U_k+U_m}{U}\bigr)^{\frac{pq}{q-p}} \+ U\rho^2 \bigr\|_1^{\frac1p-\frac1q}
  \avoidbreak
\]
for all $\lambda \ge 2\omega$ and $k,m\in\N$.

Note that $\mkU\rho^2\in L^1$ since $\mkU \le c(h_0 + 1)$ and $\rho\in D(h_0)$.
By the dominated convergence theorem it follows that the sequence
$\bigl(\lambda (\lambda + \gen_{U_m})^{-1}f\bigr)_{m \in \N}$
is convergent in~$L^p$, uniformly for $\lambda \ge 2\omega$. 
Hence there exists
a $C_0$-semigroup $S_p$ on~$L^p$ such that $S_{U_m\shift,p}\to S_p$ in the semigroup sense.
Since, by Proposition~\ref{appr}, the semigroups $S_{U_m\shift,p}$ are analytic
in a common sector and with a uniform bound, we conclude that $S_p$ is analytic as well.
Moreover, \eqref{contr-p} implies the bound~\eqref{q-c-main2}.

Finally, let $r\in J\setminus \inter{J}$. Then it follows from~\eqref{contr-p}
that $S_p$ extrapolates to a quasi-contractive semigroup $S_r$ on $L^r$,
which is a $C_0$-semigroup by \cite{voi92}.
\end{proof}

The above proof also gives the following estimate that will be important
in the next two sections.

\begin{theorem} \label{pcomplex330}
Let Assumptions~\ref{assump1}~--~\ref{assump8} be satisfied,
let $p\in \sinter J$ and let $\gen_p$ be as in Theorem~\ref{mainresult2}.
Then for all $u\in D(\gen_p)$ one has $u|u|^{\frac p2-1}\in D(h_0)$ and 
\[
  \RRe\dual{\gen_p u}{u|u|^{p-2}} \ge \h_p\bigl(u|u|^{\frac p2-1}\bigr).
\]
\end{theorem}
\begin{proof}
We use the notation from the previous proof.
By Proposition~\ref{appr} one has
\[
  \RRe\dual{\gen_{U_m\shift,p}u}{u|u|^{p-2}}
  \ge \h_p\bigl(u|u|^{\frac p2-1}\bigr)
\]
and hence $u|u|^{\frac p2-1}\in D(h_0)$ for all $m \in \N$ and $u \in D(\gen_{U_m, p})$.
Then the assertion is a consequence of Lemma~\ref{srs}\ref{srs(a)}.
\end{proof}

\section{Proof of Theorem~\ref{mainresult}}\label{S-mainresult}

In this section we derive Theorem~\ref{mainresult} from our more general result,
Theorem~\ref{mainresult2}.
The main technical part of the proof of Theorem~\ref{mainresult} is contained in the next lemma. 
Recall the definitions of $\delta_p$, $\eps_p$, $\homega_p$, $\Bhat_p$ and $I$ from 
\eqref{deo}~--~\eqref{I-def}.

\begin{lemma}\label{calculus}
Let Assumptions \ref{assump1}~--~\ref{assump4} be satisfied, and let $p \in I$.
Suppose that $\beta' > 0$ or $\alpha_\s B' = 0$.
Then
\begin{equation}\label{tau-bound-below}
\tau_p(v)
\ge \bigl( \eps_p - \eps - \tfrac{\eps}{1-\eps}\delta_p^2 \bigr) \+ h_0(|v|)
    + \eps h_0(v)
    - \bigl( \homega_p + \tfrac{\eps}{1-\eps}\Bhat_p \bigr) \|v\|_2^2
\end{equation}
for all $v \in D(h_0)$ and $\eps \in [0,1)$.
\end{lemma}

\begin{proof}
Let $\zeta = \overline{\sgn v}\,\nabla v$, $\xi=\RRe\zeta=\nabla|v|$
and $\eta = \IIm \zeta$. Then an easy computation yields
\[
  \RRe\dual{\cA\nabla v}{\nabla v}
    - (1-\rfrac2p)^2 \dual{\cA_0^\s \nabla |v|}{\nabla |v|}
  = \tfrac4{pp'}\dual{\cA_0^\s\xi}{\xi} + \dual{\cA_0^\s\eta}{\eta}
    + 2\dual{\cA_1^\a\xi}{\eta},
\]
so
\begin{align*}
\tau_p(v) =
{} & \smash[t]{ \tfrac4{pp'}\dual{\cA_0^\s\xi}{\xi} + \dual{\cA_0^\s\eta}{\eta}
     + 2\dual{\cA_1^\a\xi}{\eta} - 2|1-\rfrac2p|\int |\<\cA_1^\s\xi, \eta\>| } \\
   & + 2\dual{|v|\RRe(\tfrac1pb_1-\tfrac1{p'}b_2)}{\xi}
     - \dual{|v|\IIm(\bplusb)}{\eta} + V^+(v) - V^-(v) \\[0.5ex]
= {} & I_1 + I_2 + \ldots + I_8.
\end{align*}

The terms
$I_3$, $I_4$ and $I_6$ involving the imaginary parts of the
coefficients can be estimated as follows.
Firstly
\[
-I_4 = 2|1-\rfrac2p| \int |\<\cA_1^\s\xi, \eta\>|
   \le 2|1-\rfrac2p| \cdot \alpha_\s \dual{\cA_0^\s\xi}{\xi}^{1/2}
                                     \dual{\cA_0^\s\eta}{\eta}^{1/2}
\]
by \eqref{alpha-s}, and secondly \eqref{elcomplex1.05(b);1} and \eqref{beta-'}
imply that
\begin{align*}
-(I_3+I_6)
 &= -2\dual{\cA_1^\a\xi}{\eta} + \dual{|v|\IIm(\bplusb)}{\eta} \\
 &= \IIm \dual{\cA_1^\a \nabla v - v\IIm (\bplusb)}{\nabla v} \alignstrut \\
 &\le \bigl( (\beta'^2h_0+B')(|v|) \bigr)^{1/2} \dual{\cA_0^\s\eta}{\eta}^{1/2}. \alignstrut
\end{align*}
Set $x = h_0(|v|) = \dual{\cA_0^\s\xi}{\xi} + V^+(v)$, $y =
\dual{\cA_0^\s\eta}{\eta}$, $z = B'\|v\|_2^2$
and $\alpha_p = \alpha_\s|1-\rfrac2p|$. Then
\begin{equation}\label{I346-est}
  {-}I_4 - (I_3+I_6)  
  \le \bigl( 2\alpha_px^{1/2} + (\beta'^2x+z)^{1/2}\bigr) y^{1/2} .
\end{equation}

Now we employ the inequality
\begin{equation}\label{elemineq}
  \bigl( \alpha\sqrt a + \sqrt{\beta^2 a + b} \+ \bigr)^2 \le (\alpha+\beta)^2a + (1+\tfrac\alpha\beta)b,
\end{equation}
valid for all $\alpha,a,b\ge0$ and $\beta>0$.
(This estimate is trivial for $b=0$; for $b>0$ it easily follows by differentiating with respect to $b$.)
In the case $\beta'>0$ it follows from \eqref{elemineq}, \eqref{deo} and~\eqref{Bhat} that
\[
  \Bigl( 2\alpha_px^{1/2} + (\beta'^2x+z)^{1/2} \Bigr)^{\!2}
  \le (2\alpha_p+\beta')^2 x +
     \biggl( \mkern-1.5mu 1 + \frac{2\alpha_p}{\beta'} \mkern-1.5mu \biggr) z
  = 4 \bigl( \delta_p^2x + \Bhat_p\|v\|_2^2 \bigr);
\]
in the case $\beta' = \alpha_\s B' = 0$ this becomes an equality
if \smash{$\frac{2\alpha_p}{\beta'}$} is replaced with~$0$ (note that then $\alpha_pz = 0$).
Thus by~\eqref{I346-est} we obtain
\[
  -I_4 - (I_3+I_6)
  \le 2 \bigl( \delta_p^2x + \Bhat_p\|v\|_2^2 \bigr)^{1/2} y^{1/2}
  \le \frac{1}{1-\eps} \bigl( \delta_p^2x + \Bhat_p\|v\|_2^2 \bigr) + (1-\eps) y.
\]

Since $\xi = \nabla |v|$, we infer from \eqref{beta-1-2} and~\eqref{gamma} that
\begin{align*}
-(I_5 + I_8)
 &= -\frac2p \dual{|v|\RRe b_1}{\xi} + \frac2{p'} \dual{|v|\RRe b_2}{\xi} + V^-(v) \\[0.5ex minus 0.2ex]
 &\le \left( \frac2p\beta_1 + \frac2{p'}\beta_2 + \gamma \right) h_0(|v|)
      + \left( \frac2p B_1 + \frac2{p'}B_2 + \Gamma \right) \|v\|_2^2\,.
\end{align*}
Moreover, $I_1 + I_7 \ge \frac{4}{p p'} x$ and $I_2 = y$, so we conclude that
\begin{align*}
\tau_p(v)
 \ge {} & {\left( \frac{4}{pp'} - \frac{1}{1-\eps} \+ \delta_p^2
      - \frac2p \beta_1 - \frac2{p'} \beta_2 - \gamma \right)}x + \eps y \\[0.2ex minus 0.2ex]
 & {} - \left( \frac{1}{1-\eps} \Bhat_p
              + \frac2p B_1 + \frac2{p'} B_2 + \Gamma \right) \|v\|_2^2  \\[0.5ex minus 0.2ex]
 = {} & \bigl( \eps_p - \eps - \tfrac{\eps}{1-\eps}\delta_p^2 \bigr) x
      + \eps (x+y)
      - \bigl( \homega_p + \tfrac{\eps}{1-\eps} \Bhat_p \bigr) \|v\|_2^2\,.
\end{align*}
This completes the proof since $x = h_0(|v|)$ and $x+y = h_0(v)$.
\end{proof}

Now we are able to prove the first main theorem of this paper.

\begin{proof}[Proof of Theorem~\ref{mainresult}]
Applying Lemma~\ref{calculus} with $\eps=0$ we obtain
\begin{equation}\label{abs-est}
  \tau_p(v) \ge \eps_p h_0(|v|) - \homega_p \|v\|_2^2
\end{equation}
for all $p\in I$ and $v\in D(h_0)$. 
It follows that $\tomega_p \le \homega_p$,
with $\tomega_p$ as in~\eqref{omegap-tilde-def},
and this implies $I\subseteq J$.
Moreover $\h_p(v) \ge \eps_p h_0(|v|) - \homega_p \|v\|_2^2$ for all $v\in D(h_0)$
since the right-hand side of~\eqref{abs-est} is lower semi-continuous in~$v$.

Next let $p\in \sinter{I}$, so that $\eps_p > 0$.
Let $\eps \in (0,1)$ satisfy
$\eps + \smash{\frac{\eps}{1-\eps}\delta_p^2} \le \eps_p$.
Then it follows from Lemma~\ref{calculus} that~\eqref{mu-omega} holds 
with $\mu_p = \eps$ and $\omega_p = \homega_p + \frac{\eps}{1-\eps}\smash{\Bhat_p}$.
In particular, Assumption~\ref{assump8} is satisfied since
$p \in \sinter I \subset \sinter J$.
Moreover, $\h_p(v) \ge \eps h_0(v) - \omega_p\|v\|_2^2$ for all $v\in D(h_0)$.
Now all the assertions of Theorem~\ref{mainresult} follow from Theorem~\ref{mainresult2}
and Theorem~\ref{pcomplex330}.
\end{proof}

The case where $\beta'=0$ and $\alpha_\s B' > 0$ was excluded in Theorem~\ref{mainresult}.
In this case one does not necessarily obtain a quasi-contractive $C_0$-semigroup on~$L^p$
corresponding to $\gen$ if $p \in \partial I$, as the following example shows.

\begin{example}\label{open-ex}
Let $N=1$, $\Omega=(0,1)$, $\cA=1+i$, $b_1=-i$, $b_2=Q=0$ and $D(h_0) = W^{1,2}(\Omega)$.
Then $\gen = -(1+i)\frac{d^2}{dx^2} - i\frac{d}{dx}$ with Neumann boundary conditions.
Observe that in \eqref{alpha-s}~--~\eqref{gamma}
one can choose $\alpha_\s = 1$, $\beta_1 = \beta_2 = B_1 = B_2 = \gamma = \Gamma = 0$, 
$\beta'=0$ and $B'=1$.
Therefore,
\[
  \eps_p = \tfrac{4}{pp'} - (1-\rfrac2p)^2 = 1 - 2(1-\rfrac2p)^2
\]
and hence $I = [4-2\sqrt2,4+2\sqrt2]$. It follows from Theorem~\ref{mainresult},
applied with small $\beta'>0$ instead of $\beta'=0$,
that $S_2$ extrapolates to a quasi-contractive $C_0$-semigroup $S_p$ on $L^p$,
for all $p\in \inter{I}$, but we now show that this is not true for $p=4+2\sqrt2$.

Let $p=4+2\sqrt2$ and $r=1-\sqrt2-i$. Let $\tau\from[0,1]\to[0,\infty)$ be a
$C^\infty$-function satisfying $\tau(0) = \tau'(0) = \tau'(1) = 0$ and $\tau > 0$ on $(0,1]$.
Let $\lambda>0$ and consider $u = u_\lambda = e^{\lambda r\tau}$.
Then $u \in C^\infty[0,1]$ and $\frac{\partial u}{\partial n} = 0$ on $\partial\Omega$,
so $u \in D(\gen_2)$. We are going to compute the real part of
\[
  \dual{\gen_2 u}{|u|^{p-2}u}
  = (1+i) \dual{u'}{(|u|^{p-2}u)'} - i\dual{u'}{|u|^{p-2}u}.
\]
We have $u' = \lambda r\tau'u$ and $|u|^{p-2}u = \exp\bigl(\lambda(r + (p-2)\RRe r)\tau\bigr)$.
Therefore,
\begin{align*}
\inner{u'}{(|u|^{p-2}u)'}
 &= \inner{\lambda r\tau'u}{\lambda\bigl(r + (p-2)\RRe r\bigr)\tau'|u|^{p-2}u}
    \\[0.5ex]
 &= \lambda^2\bigl(|r|^2 + (p-2)r\RRe r\bigr)(\tau')^2|u|^p.
\end{align*}
A straightforward computation yields $|r|^2 + (p-2)r\RRe r = 2(1+i)$, so we infer that
\[
  \dual{\gen_2 u}{|u|^{p-2}u}
  = 4i\lambda^2 \int_0^1 (\tau')^2|u|^p - i\lambda r \int_0^1 \tau'|u|^p.
\]
Since $\RRe(ir) = 1$ and $|u|^p = e^{p\lambda(\RRe r)\tau} = e^{-\sqrt8 \lambda\tau}$,
we conclude that
\[
  \RRe \dual{\gen_2 u_\lambda}{|u_\lambda|^{p-2}u_\lambda}
  = -\lambda \int_0^1 \tau'|u|^p
  = \frac{1}{\sqrt8} e^{-\sqrt8 \lambda\tau} \big|_0^1
  \to -\frac{1}{\sqrt8} \qquad (\lambda\to\infty).
\]
Because $u_\lambda \in L^p$, $\gen_2u_\lambda \in L^p$
and $\|u_\lambda\|_p \to 0$ as $\lambda \to \infty$, the Lumer--Phillips theorem implies
that $S_2$ does not extrapolate to a quasi-contractive $C_0$-semigroup $S_p$ on~$L^p$.
\end{example}

\begin{remark}\label{const-rem2}
Here we comment on the inequalities~\eqref{alpha-s}~--~\eqref{beta-1-2}
in Lemma~\ref{lcomplex1.05}.
\begin{alplist}[i]
\item \label{const-rem(d)}
It is easy to see that~\eqref{beta-1-2} holds if
$\<(\cA_0^\s)^{-1} b_j, b_j\> \le \beta_j^2\+h_0 + 2\beta_jB_j$.
Such an assumption was used, e.g., in \cite{L} to obtain quasi-contractive
semigroups as in Theorem~\ref{mainresult}.

\item \label{const-rem(a)}
From the definition $\eta(u) = \IIm(\overline{\sgn u}\,\nabla u)$ it follows that
$\eta(\overline u) = -\eta(u)$ for all $u \in D(h_0)$.
Thus, \eqref{beta-'} is actually a two-sided estimate:
\[
  \bigl| \IIm \dual{\cA_1^\a \nabla u -  u\IIm(\bplusb)}{\nabla u} \bigr|
  \le \bigl(\beta'^2h_0(|u|)+B'\|u\|_2^2\bigr)^{\frac12} \dual{\cA_0^\s \eta(u)}{\eta(u)}^{\frac12}.
\]

\item \label{const-rem(c)}
By Assumptions~\ref{assump3} and~\ref{assump4},
there exist constants $\alpha_\a, \hat B\ge0$ and $\hat\beta>0$ such that
\vspace{-1.5ex}
\begin{equation}\label{alpha-a-hat-beta}
\begin{split}
|\< \cA_1^\a \xi,\eta\>|^2
  &\le \alpha_\a^2 \+ \<\cA_0^\s \xi,\xi\> \+ \<\cA_0^\s \eta,\eta\>, \\[0.8ex minus 0.3ex]
\inner{(\cA_0^\s)^{-1}\IIm(\bplusb)}{\IIm(\bplusb)}
  &\le \hat\beta^2 \+ h_0+\hat B
\end{split}
\end{equation}
for all $\xi,\eta\in\R^N$.
These inequalities are somewhat more explicit than~\eqref{beta-'};
we now show that they imply~\eqref{beta-'} with
$\beta'= 2 \alpha_\a + \hat\beta$ and $B'= (1 + \frac{2\alpha_\a}{\hat\beta}) \hat B$.

Let $u\in D(h_0)$ and set $\xi = \nabla|u|$, $\eta = \IIm(\overline{\sgn u}\,\nabla u)$.
Then \eqref{elcomplex1.05(b);1} gives
\begin{align*}
\IIm \dual{\cA_1^\a \nabla u - u\IIm(&\bplusb)}{\nabla u}
    = \dual{-2\cA_1^\a \xi + |u|\IIm(\bplusb)}{\eta} \\[0.3ex minus 0.2ex]
 &\le \Bigl( 2\alpha_\a (\cA_0^\s\xi,\xi)^\frac12
             + \bigl( \hat\beta^2 \+ h_0(|u|) + \hat B\|u\|_2^2 \bigr)^\frac12 \Bigr)
      (\cA_0^\s\eta,\eta)^\frac12.
\end{align*}
Since $(\cA_0^\s\xi,\xi) \le h_0(|u|)$, it follows from~\eqref{elemineq} that
\[
  \IIm \dual{\cA_1^\a \nabla u -  u\IIm(\bplusb)}{\nabla u}
  \le \Bigl((2 \alpha_\a + \hat\beta)^2h_0(|u|)+(1 + \tfrac{2\alpha_\a}{\hat\beta})\hat B\|u\|_2^2\Bigr)^{\!\frac12} \mkern-1mu
      \dual{\cA_0^\s \eta}{\eta}^{\frac12},
  \avoidbreak
\]
i.e., the estimate~\eqref{beta-'} holds with
$\beta'= 2 \alpha_\a + \hat\beta$ and $B'= (1 + \frac{2\alpha_\a}{\hat\beta}) \hat B$.
\end{alplist}
\end{remark}

In the next example we explain why the estimates~\eqref{alpha-a-hat-beta}
are much cruder than~\eqref{beta-'}.
The main point is that $\cA_1^\a$ essentially plays the role of a
first-order coefficient $i\div\cA_1^\a$, rather than a second-order coefficient.

\pagebreak[2]

\begin{example}\label{imag-coeffs}
Let us assume that $C_\c^\infty(\Omega)$ is a core for~$h_0$,
i.e., $h_0=\overline{h_{\max}\restrict{C_\c^\infty(\Omega)}}$.
\begin{alplist}[i]
\item \label{imag-coeffs(a)} 
Suppose that
$b := i\div\cA_1^\a = \bigl(i\sum_{j=1}^N \partial_j(\cA_1^\a)_{jk}\bigr)_k$
is a locally integrable vector field. Then by the anti-symmetry of $\cA_1^\a$ one obtains
\begin{align*}
\IIm \dual{\cA_1^\a \nabla u - u\IIm(\bplusb)}{\nabla u}
 &= \IIm \dual{u\IIm (b - b_1 - b_2)}{\nabla u} \\
 &= \dual{|u|\IIm (b - b_1 - b_2)}{\eta(u)}
\end{align*}
for all $u\in C_\c^\infty(\Omega)$ and hence for all $u\in D(h_0)$.
Thus, \eqref{beta-'} is valid if
\[
  \inner{(\cA_0^\s)^{-1}\IIm (b - b_1 - b_2)}{\IIm (b - b_1 - b_2)}
  \le \beta'^2 \+ h_0 + B'.
\]
Note that $b=0$ if $\cA_1^\a$ has divergence free columns (e.g., if $\cA_1^\a$ is constant).
In this case, the form $\form$ and the associated semigroup are independent of $\cA_1^\a$,
and the constant $\alpha_\a$ in~\eqref{alpha-a-hat-beta} does not play any role.

\item \label{imag-coeffs(b)} 
Assume that $b := \IIm(\bplusb) \in L^p(\Omega)$ for some $p>1$
and that $b$ is the restriction of a
compactly supported divergence free vector field $b^\sharp\in L^p(\R^N)$.
If $b$ itself is divergence free, then this is the case, for instance,
when $b$ has compact support in $\Omega$ or when $\Omega$ is a simply connected Lipschitz domain.
(In the latter case use \cite[Proposition~4.1]{kmpt00} with $k=s=0$ and $\sigma=1$.)
Then the exterior derivative of $(-\Delta)^{-1}b^\sharp$,
\[
  \cA' := \Bigl(\tfrac12\bigl(\partial_k(-\Delta)^{-1}b^\sharp_j-\partial_j(-\Delta)^{-1}b^\sharp_k\bigr)\Bigr)_{\mkern-1mu\raisebox{0.3ex}{$\scriptstyle jk$}}
      \in W^{1,p}(\R^N)
\]
is an anti-symmetric matrix such that $\div \cA' = b^\sharp$.
As in~\ref{imag-coeffs(a)} it follows that
\[
  \dual{\cA_1^\a \nabla u - u\IIm(\bplusb)}{\nabla u}
  = \dual{(\cA_1^\a - \cA')\nabla u}{\nabla u}
\]
for all $u\in D(h_0)$.
Hence, if there exists an $\alpha\ge0$ such that
\begin{equation}\label{cA'}
  \bigl|\inner{(\cA_1^\a - \cA')\xi}{\eta}\bigr|^2
  \le \alpha^2 \<\cA_0^\s \xi,\xi\> \+ \<\cA_0^\s \eta,\eta\>
\end{equation}
for all $\xi,\eta\in \R^N$, then by~\eqref{elcomplex1.05(b);1} we see that
\eqref{beta-'} is valid with $\beta'=2\alpha$ and $B'=0$.
In the case where $\IIm(\bplusb)$ is oscillating, this can lead to much better
estimates than Remark~\ref{const-rem2}\ref{const-rem(c)}.
\end{alplist}
\end{example}

\section{Extension of the interval}\label{extension}

According to Theorem~\ref{mainresult2}, the couple $(\form,\U)$ is associated with
a quasi-contractive $C_0$-semigroup on $L^p$ for every $p\in J$.
The aim in this section is to show that under additional assumptions
there exist consistent analytic $C_0$-semigroups on $L^p$ for $p$ from a larger interval. 
In general, these semigroups are not quasi-contractive any more.
\pagebreak[2]

The following result is an extension of \cite[Proposition~9]{lv00}.

\begin{lemma}\label{bdd-multipl}
Suppose that $D(h_0)$ is an ideal of $D(\amax)$.
Let $\rho\in W^{1,1}_\loc \cap L^\infty$, and
assume that there exists a $c\ge0$ such that
$\<\cA_0^\s \nabla \rho, \nabla \rho\> \le c(h_0+1)$.
Then $\rho$ is a bounded multiplication operator on $D(h_0)$.
\end{lemma}
\begin{proof}
Firstly, observe that $\rho$ is a bounded multiplication operator on $\Q(V^+)$.
Secondly, let $u\in D(h_0)\cap L^\infty$.
Then $\rho u\in W^{1,1}_\loc$ and
\begin{equation}\label{eq-bdd-multipl}
\int \<\cA_0^\s \nabla(\rho u),\nabla (\rho u)\>
\le 2c \bigl( h_0(u) + \|u\|_2^2 \bigr) + 2 \|\rho \|_\infty^2 h_0(u).
\end{equation}
So $\rho u \in D(\amax)$ and $|\rho u|\le \|\rho\|_\infty |u|$.
By the ideal property we obtain $\rho u\in D(h_0)$.
Since $D(h_0)\cap L^\infty$ is a core for $h_0$,
the assertion now follows from~\eqref{eq-bdd-multipl}.
\end{proof}

\begin{remark}
By a similar argument as above one can show the following converse of Lemma~\ref{bdd-multipl}:
If $\rho \in W^{1,1}_\loc \cap L^\infty$
is a bounded multiplication operator on $D(h_0)$, then there exists a $c\ge0$ such that
$\inner{\cA_0^\s \nabla\rho}{\nabla\rho} \le c(h_0+1)$.
For this one does not need to assume that $D(h_0)$ is an ideal of $D(\amax)$.
\end{remark}

\begin{proof}[Proof of Theorem~\ref{thm-extension}]
We prove the assertion for all $q\in(p_-,p_{\max})$;
then by duality it also holds for all $q\in(p_{\min},p_+)$.
Fix $p\in(p_-,p_+)$. 
We shall show that there exist $M,\mu,\omega\ge0$ such that
\begin{align}
\|\rho_\xi^{-1} S_p(t) \rho_\xi\|_{p\to p}
 &\le Me^{\mu|\xi|^2t+\omega t}, \label{ext1} \\[0.5ex]
\|S_p(t)\|_{p\to \frac{N}{N-2}p}
 &\le Mt^{-1/p}e^{\omega t} \label{ext2}
\end{align}
for all $t>0$ and $\xi\in\R^N$, where $\rho_\xi(x) := e^{-\langle\xi, x\rangle}$;
then \cite[Proposition~2.8]{lsv02} yields the assertion for all
$q\in[p,\frac{N\vphantom{N^t}}{N-2}p)$.

For the proof of~\eqref{ext1} fix $\xi\in\R^N$ and $n\in\N$,
and set $\rho_n=\rho_\xi\wedge n$. Then
\[
  W_{\lrho_n}
  := \frac{\<\cA_0^s\nabla \rho_n, \nabla \rho_n\>}{\rho_n^2}
  = \<\cA_0^s\xi,\xi\>\ind_{[\rho_\xi\le n]}
  \le c_2 |\xi|^2 =: K
\]
with $c_2$ as in assumption~\eqref{uell}.
In particular, $\rho_n$ is a bounded multiplication
operator on $D(h_0)$, by Lemma~\ref{bdd-multipl}.
Applying Proposition~\ref{weight-Um}\ref{weight-Um(b)} with $\eps=1$,
one finds $\mu,\omega>0$, independent of $\xi$ and $n$, such that
\[
  \|\rho_n^{-1}S_\Up(t)\rho_n\|_{p\to p}
  \le \exp\bigl((\omega+\mu|\xi|^2)t\bigr)
\]
for all $t\ge0$ and $\mkU\in\U$.
Now let $f\in L^p$ be such that $\rho_\xi f\in L^p$.
Choosing $U = (\smash{\mkUhat} - m)^+$ and passing to the limit $m \to \infty$,
we infer from the previous estimate that
\[
  \|\rho_n^{-1}S_p(t)\rho_n f\|_{p\to p}
  \le \exp\bigl((\omega+\mu|\xi|^2)t\bigr) \|f\|_p
  \avoidbreak
\]
for all $t\ge0$.
Then by Fatou's lemma we conclude that~\eqref{ext1} holds with $M=1$.

By Theorem~\ref{pcomplex330}, the estimate \eqref{mu-omega} and the assumption
$D(h_0) \subset L^{\frac{2N}{N-2}}$ (combined with the closed graph theorem),
there exists a $\delta>0$ such that
\[
  \RRe\dual{(\omega_p+1+\gen_p)u}{u|u|^{p-2}}
  \ge \mu_p h_0(u|u|^{p/2-1}) + \| \+ |u|^{p/2}\|_2^2
  \ge \delta\|u\|_{\vphantom{N^t}\smash{\frac{N}{N-2}}p}^p
\]
for all $u\in D(\gen_p)$.
Now let $f\in L^p$, and set $u_t = \smash{e^{-t(\omega_p+1+\gen_p)}f}$ for all $t>0$.
Using H\"older's inequality and the analyticity of $S_p$ one deduces
that there exists a $C>0$ such that
\[
  \delta\|u_t\|_{\frac{N}{N-2}p}^p
  \le \RRe\dual{(\omega_p+1+\gen_p)u_t}{u_t|u_t|^{p-2}}
  \le \frac{C}{t} \|f\|_p^p
\]
for all $t>0$. It follows that
$\smash{\|e^{-t(\omega_p+1+\gen_p)}f\|_{\frac{N}{N-2}p}}
\le \bigl(\frac{C}{\delta t}\bigr)^{-1/p}\|f\|_p$ for all $t>0$,
which proves~\eqref{ext2}.
\end{proof}

\begin{remark} \label{rcomplex504}
In the case $N \le 2$ a similar result can be proved
if one replaces the assumption of the Sobolev embedding theorem
with the following Gagliardo--Nirenberg inequality: for all $r\in (2,\infty]$
with $\theta := \frac{N}{2}(1-\rfrac2r) < 1$ there exists a $c>0$ such that
\[
  \|v\|_r \le c [(h_0+1)(v)]^{\theta/2}\|v\|_2^{1-\theta}
\]
for all $v\in D(h_0)$.
Then a minor adaptation of the above proof shows that in the case $N=2$ the semigroup
$S_p$ extrapolates to an analytic $C_0$-semigroup on $L^q$ for all $q\in(1,\infty)$;
in the case $N=1$ the semigroup $S_p$ extrapolates to an analytic $C_0$-semigroup on $L^q$
for all $q\in[1,\infty)$, and the integral kernel satisfies Gaussian upper bounds.
For $N=1$ see also \cite{da95}.
\end{remark}

\section*{Acknowledgment}
Parts of this paper were written whilst the first and the fourth named authors visited Swansea University
and the second named author visited The University of Auckland. The authors are grateful to these institutions for the support.
In addition the first named author wishes to thank the LMS for financial support by the LMS visitors program.
Part of this work is supported by the Marsden Fund Council from
Government
funding, administered by the Royal Society of New Zealand.

The authors thank J\"urgen Voigt for communicating the proof of Lemma~\ref{dense}.

\bigskip
\bigskip

\noindent
\parbox[t]{7.5cm}{\small
A.F.M. ter~Elst\\
Department of Mathematics\\
The University of Auckland\\
Private bag 92019\\
Auckland 1142\\
New Zealand\\
{\tt terelst@math.auckland.ac.nz}
}
\parbox[t]{6.5cm}{\small
Vitali Liskevich\\
Department of Mathematics\\
Swansea University\\
Singleton Park, Swansea, SA2 8PP\\
Wales, UK
}

\bigskip
\bigskip

\noindent
\parbox[t]{7.5cm}{\small
Zeev Sobol\\
Department of Mathematics\\
Swansea University\\
Singleton Park, Swansea, SA2 8PP\\
Wales, UK\\
{\tt z.sobol@swansea.ac.uk}
}
\parbox[t]{6.5cm}{\small
Hendrik Vogt\\
Fachbereich 3 -- Mathematik\\
Universit\"at Bremen\\
Postfach 330 440\\
28359 Bremen, Germany\\
{\tt 
hendrik.vo\rlap{\textcolor{white}{hugo@egon}}gt@uni-\rlap{\textcolor{white}{%
hannover}}bremen.de}
}

\end{document}